\documentclass[12pt,leqno,fleqn]{amsart}
\usepackage[colorlinks,citecolor=red,pagebackref,hypertexnames=true]{hyperref}
\usepackage[backrefs,msc-links,nobysame,initials]{amsrefs}
\usepackage[normalem]{ulem}
\usepackage{amsmath,amstext,amsthm,amssymb,amsxtra}
\usepackage{txfonts,pxfonts} 
\usepackage{pgf,tikz}
\usepackage[color=green!15,bordercolor=red,linecolor=red!30]{todonotes}
\usepackage{graphicx}
\usepackage{transparent}
\usepackage{enumitem}
\usepackage{calc}
\usepackage{mathtools}
\mathtoolsset{showonlyrefs,showmanualtags}

\allowdisplaybreaks

\swapnumbers 

\numberwithin{equation}{section}

\theoremstyle{plain}
\newtheorem{theorem}[equation]{Theorem}
\newtheorem{proposition}[equation]{Proposition}

\newtheorem{lemma}[equation]{Lemma}

\theoremstyle{definition}

\theoremstyle{remark}
\newtheorem{remark}[equation]{Remark}

\makeatletter
\@namedef{subjclassname@2010}{%
  \textup{2010} Mathematics Subject Classification}
\makeatother

\def\norm#1.#2.{\lVert#1\rVert_{#2}}
\def\Norm#1.#2.{\bigl\lVert#1\bigr\rVert_{#2}}
\def\NOrm#1.#2.{\Bigl\lVert#1\Bigr\rVert_{#2}}
\def\NORm#1.#2.{\biggl\lVert#1\biggr\rVert_{#2}}
\def\NORM#1.#2.{\Biggl\lVert#1\Biggr\rVert_{#2}}

\def\ip#1,#2,{\langle #1,#2\rangle}
\def\Ip#1,#2,{\bigl\langle#1,#2\bigr\rangle}
\def\IP#1,#2,{\Bigl\langle#1,#2\Bigr\rangle}

\def\abs#1{\lvert#1\rvert}
\def\Abs#1{\bigl\lvert#1\bigr\rvert}
\def\ABs#1{\biggl\lvert#1\biggr\rvert}

\def\N{\mathbb N}
\def\Z{\mathbb Z}
\def\R{\mathbb R}
\def\RR{\mathbb R _+}

\def\bmo{\textnormal{BMO}}

\def\P{\mathbf P}
\def\T{\mathbf T}

\def\size{\operatorname{size}}
\def\density{\operatorname{density}}
\def\vr{\mathcal{V}^r}

\def\eqdef{\stackrel{\mathrm{def}}{{}={}}}
\def\iffdef {\stackrel{\mathrm{def}}{{}\Leftrightarrow{}}}

\def\ind{\textnormal{\textbf 1}}
\newcommand{\Exp}[0]{\mathbb{E}}
\begin{document}
\title[A variation norm Carleson theorem for vector-valued Walsh-Fourier series]{A variation norm Carleson theorem for vector-valued Walsh-Fourier series}

\author[T.~Hyt\"onen]{Tuomas P. Hyt\"onen}
\address{Department of Mathematics and Statistics, P.O.B.~68 (Gustaf H\"all\-str\"omin katu~2b), FI-00014 University of Helsinki, Finland}
\email{tuomas.hytonen@helsinki.fi}
\thanks{T.H. and I.P. are supported by the European Union through the ERC Starting Grant ``Analytic-probabilistic methods for borderline singular integrals''.
T.H. is also supported by the Academy of Finland, grants 130166 and 133264 and I.P. also supported by the Academy of Finland, grant 138738.}

\author[M.T. Lacey]{Michael T. Lacey}
\address{School of Mathematics,
Georgia Institute of Technology,
Atlanta GA 30332 }
\email{lacey@math.gatech.edu}
\thanks{M.L. supported in part by the NSF grant 0968499, and a grant from the Simons Foundation (\#229596 to Michael Lacey). }

\author[I. Parissis]{Ioannis Parissis}
\address{Department of Mathematics, P.O.B~ 11100 (Otakaari 1 M), FI-00076 Aalto University, Finland}
\email{ioannis.parissis@gmail.com}


\begin{abstract}
We prove a variation norm Carleson theorem for Walsh-Fourier series of functions with values in certain UMD Banach spaces, sharpening a recent result of Hyt\"onen and Lacey. They proved the pointwise convergence of Walsh-Fourier series of $X$-valued functions under the qualitative hypothesis that $X$ has some finite tile-type $q<\infty$, which holds in particular if $X$ is intermediate between another UMD space and a Hilbert space. Here we relate the precise value of the tile-type index to the quantitative rate of convergence: tile-type $q$ of $X$ is ``almost equivalent'' to the $L^p$-boundedness of the $r$-variation of the Walsh-Fourier sums of any function $f\in L^p([0,1);X)$ for all $r>q$ and large enough $p$. 
\end{abstract}

\maketitle

\section{Introduction} The celebrated theorem of Carleson on the pointwise convergence of Fourier series, \cite{Car}, asserts that the partial sums of the Fourier series of an $L^p$-function converge almost everywhere to the function for $1<p<\infty$. The usual strategy to prove this result is to show that the \emph{Carleson operator}, that is, the maximal partial sums of the Fourier series of a function, is bounded on the corresponding $L^p$-space. Subtle refinements of the Carleson theorem \cite{CLTT} involve more refined notions than the $\ell^\infty$-norm of the partial sums. One of these notions is the variation norm Carleson theorem, proved in \cite{OSTTW}, which in particular shows the pointwise convergence of the Fourier series of $f$ without having to resort to a dense subset where this convergence can be easily exhibited. Weighted versions of the corresponding results for Fourier and Walsh-Fourier series have recently appeared in \cites{DOLA,DOLA1}.

We point out that, in the scalar case, variational  bounds for the Hilbert transform as well as for more general Calder\'on-Zygmund operators have been obtained for example in \cite{CJRW} and \cite{CJRW1}. In \cite{JSW} the authors study more singular operators such as averages along lower dimensional sets and truncations of singular integrals with rough kernels. However, there is no complete characterization of the Calder\'on-Zygmund kernels, say, that give rise to operators that obey variational bounds. The question is subtle since a variational bound immediately implies the pointwise converge of the truncations of singular integrals and it is known that this convergence fails for some operators\footnote{The study of variational bounds for quite general Calder\'on-Zygmund operators is addressed systematically in the paper \cite{TTX} which appeared while the current paper was under review.}. In the vector-valued setup the information is even scarcer. For example it is natural to ask whether the Hilbert transform of a function with values in a UMD Banach space obeys variational bounds. That would be the vector-valued analog of the main result in \cite{CJRW} and, to the best of our knowledge, it is not currently known.

A necessary condition for the validity of Carleson's theorem for functions with values in a Banach space is known to be that the target Banach space $X$ has the UMD property: $X$-valued martingale differences converge unconditionally in the space. See for example \cites{HL1,Rdf:LNM,Rdf:Studia}.  The first results on Carleson's theorem for vector-valued functions appear in \cite{Rdf:LNM} and \cite{Rdf:Studia} with the additional hypothesis that the target Banach space $X$ is a UMD space with an unconditional basis, or more generally, that it is a UMD lattice. The corresponding result for Walsh-Fourier series was proved for UMD lattices  in \cite{Weisz:Vilenkin}. Recently in \cite{PSX}, Parcet, Soria and Xu  proved the weaker statement that the partial Fourier series of $f\in L^p(\mathbb T;X)$, $p>1$, satisfy $S_Nf(x)=o(\log\log N)$ for a.e. $x$, in general UMD spaces $X$, thus taking a step away from the lattice hypothesis. The first two authors of the current paper finally proved the vector valued Carleson theorem for all known examples of UMD spaces in \cite{HL1}. Besides the UMD hypothesis, which is necessary, the proof in \cite{HL1} is based on a certain key assumption on the UMD space, namely that it has \emph{finite tile-type}. The \emph{tile-type} assumption is a hypothesis of probabilistic flavor on the geometry of $X$, which is easily seen to be stronger than usual cotype, although we do not know at this moment if it is strictly stronger than the UMD property, which is known to imply finite cotype.

The tile-type hypothesis has been introduced in \cite{HL} where it is shown that finite-tile type implies that the partial sums of Walsh-Fourier series converge to the function almost everywhere. Another hypothesis of similar flavor has been used in \cite{HL1} to show the corresponding result in the trigonometric case. In these two works, the prototypical examples of UMD Banach spaces with finite tile-type are \emph{the intermediate spaces} $X=[Y,H]_\theta$, that is, spaces $X$ which are complex interpolation spaces between some UMD Banach space $Y$ and a Hilbert space $H$. This class of Banach spaces includes for example all the UMD lattices but also the Schatten ideals $C_p$, $1<p<\infty$, and in general all examples of UMD spaces that are currently known, which are not necessarily function lattices. However the tile-type turns out to be a qualitative hypothesis for the Carleson theorem: the exact value of the tile-type is irrelevant, it is only needed that it is finite. This changes dramatically when one looks at models for vector-valued bilinear Hilbert transforms, \cite{HLP}. There the exact value of the tile-types of the Banach spaces involved play a crucial role in the range of inequalities one can prove, and in fact one needs to stay `close enough' to the Hilbert space tile-type $q=2$ to get \emph{any} bounds on the Walsh model of the bilinear Hilbert transform\footnote{There is another striking detail, that one can formulate vector-valued results that include non-UMD Banach spaces; see Silva \cite{Silva:BP}.}.

In this paper we take up the investigation initiated in \cite{HL} and continued in \cites{HL1,HLP}, concerning the vector-valued extensions of Carleson's theorem and related issues, where the lattice assumption is completely avoided. In particular, this paper is very much in the spirit of \cite{HL} since we study the variant of Carleson's theorem, due to Billard \cite{Billard}, for Walsh-Fourier series. Building on the results from \cites{HL,HL1}, our purpose is to investigate whether the finite tile-type hypothesis is also \emph{necessary} for the validity of a vector-valued Carleson theorem. We provide a step towards this direction by characterizing the UMD Banach spaces for which the finite tile-type assumption is true in terms of the boundedness of a variational Carleson operator. In particular we show that a variation norm version of Carleson's theorem for vector-valued Walsh-Fourier series, in the spirit of \cites{OSTTW,DOLA}, is valid if and only if the UMD Banach space satisfies the finite tile-type hypothesis. 

Furthermore, we manage to quantify the relation between the tile-type of the Banach space  and the variation index in the variational Carleson theorem. Although this is not possible in the usual formulation of the Carleson theorem, the variational variant gives the correct framework for such a quantification. Thus the variation index, which can be viewed as a quantification of the rate of convergence of the partial Walsh-Fourier series, is intimately related to the tile-type of the UMD Banach space under consideration.  We know from \cite{HL1} that tile-type implies Carleson's theorem. However we now see that the tile type condition is essentially characterized by the validity of a variation norm Carleson theorem on a UMD Banach space. The question whether Carleson's theorem is valid on an arbitrary UMD Banach space remains open however.  Our method and general strategy of proof is along the lines of \cites{HL,DOLA,OSTTW}, using the time-frequency analysis techniques and arguments introduced in \cite{LT:MRL}.

The first thing we need to do in order to formulate our main results is to describe what is the variation norm of a sequence. This norm plays a central role throughout the paper.

Let $1\leq r \leq \infty$ and $x=\{x_n\}_{n\in\mathbb N}$ be a sequence of elements of a Banach space $(X,|\cdot|)$. The $r$-variation of the sequence $x$ is defined as 
\begin{align*}
\norm x.\vr(X) . \coloneqq\sup_K \sup_{ N_0 < \cdots < N_K}\bigg( \sum_{j=1} ^{K} \abs {x_{N_j}-x_{N_{j-1}}} ^r\bigg)^\frac{1}{r}.
\end{align*}
with the usual modification when $r=\infty$. Observe that the $\mathcal V^\infty$-norm is essentially the $\ell^\infty$-norm.
If $\{f_n\}_{n\in\mathbb N}$ is a sequence of $X$-valued functions, $f_n:\RR\to X$, we will also write
\begin{align*}
\norm f_n.L^p(\RR;\vr(X)). \coloneqq \bigg(\int_{\RR} \norm f _n (x).\vr(X). ^p dx\bigg)^\frac{1}{p}.
\end{align*}

Now for $f\in L^p([0,1);X)$ for some $1< p <\infty$ consider the partial Walsh sums
\begin{align*}
S_Nf(x)\coloneqq \sum_{n=0} ^{N-1} \ip f,w_n,w_n(x),
\end{align*}
where $w_n$ is the $n$-th Walsh function on $[0,1)$. Precise definitions will be given in Section \ref{s.packets}.
The variational Carleson operator is the non-linear operator $\|S_N f(x)\|_{\vr}$. As we shall see in detail in \S~\ref{s.linearization} it can be controlled in the form
\begin{align*}
 \|S_N f(x)\|_{\vr}\lesssim C_{r,\P_1} f(x) + \tilde C_{r,\P_2}f(x)
\end{align*}
where $C_{r,\P_1}, \tilde C_{r,\P_2}$ are two linear operators defined in \eqref{e.crp} in terms of notions from time-frequency analysis which will be introduced in \S~\ref{s.packets}. The point of introducing these operators is that they linearize and control the variational Carleson operator, and are more amenable to time-frequency analysis. Our main theorem characterizes the tile-type of a Banach space in terms of the boundedness of the operator $C_{r,\P_1}$ or the operator $\tilde C_{r,\P_2}$.
\begin{theorem}\label{t.main_new} Let $q\in[2,\infty)$ and $X$ be a Banach space. Suppose that the operator $C_{r,\P} $, or the operator $\tilde C_{r,\P}$, satisfies
\begin{align}
 \norm C_{r,\P} f. L^r(\RR;X). \lesssim \norm f.L^r(\RR;X). ,
\end{align}
whenever $q<r< \infty$. Then $X$ has tile-type $\tau$ for all $\tau>q$ and, a fortiori, $X$ has cotype $\tau$ for all $\tau>q$. Conversely, suppose that the Banach space $X$ has tile-type $\tau$ for all  $\tau >q $. Then for every $f\in L^p(\R_+;X)$,
\begin{align*}
	\norm   C_{r,\P_1} f. L^p(\R_+;X).+ \norm \tilde C_{r,\P_2} f. L^p(\R_+;X).\lesssim \norm f.L^p(\R_+;X).,\quad \text{for}\quad \max(q,p'(q-1))<r<\infty,
\end{align*}	
with the implicit constant depending only upon $p,r,q$ and the space $X$.
\end{theorem}

As a corollary we conclude a variational norm version of Carleson's theorem for vector-valued Walsh-Fourier series:
\begin{theorem}\label{t.main1}
Given $q\in[2,\infty)$, let $X$ be a Banach space which has tile-type $\tau$ for all  $\tau >q $. Suppose that
\begin{align*}
  	\max(q,p'(q-1))<r<\infty.
\end{align*}
Then for every $f\in L^p([0,1);X)$,
\begin{align*}
	\norm S_N f. L^p([0,1);\vr(X)).\lesssim \norm f.L^p([0,1);X).,
\end{align*}	
with the implicit constant depending only upon $p,r,q$ and the space $X$.
\end{theorem}
Observe that for $q=2$ the restriction on the integrability exponent becomes $p>r'$ which is necessary in the scalar case of the Fourier analog of Theorem \ref{t.main1}, \cite{OSTTW}. Here we get a condition that becomes more stringent as $X$ `moves away' from the Hilbert space case, quantified by the tile-type.

In the special case that $X$ is an \emph{intermediate space}, that is $X=[Y,H]_\theta$ is a complex interpolation space between some UMD Banach space $Y$ and a Hilbert space $H$, and $0<\theta<1$, it was shown in \cite{HL} that $X$ has tile-type $q=2/\theta$. Thus Theorem \ref{t.main1} immediately applies to all intermediate spaces of this type. However, arguing directly by interpolation we get a slightly stronger theorem:
\begin{theorem}\label{t.main2} Suppose that $X\coloneqq [Y,H]_\theta$, $0<\theta<1$, is a complex interpolation space between a UMD Banach space $Y$ and a Hilbert space $X$. Set $q\coloneqq 2/\theta$. Suppose that
\begin{align*}
     \max(q,p'q/2)<r<\infty.
\end{align*}
Then for every $f\in L^p([0,1);X)$ we have that
\begin{align*}
\norm S_N f . L^p( [0,1);\vr(X)).\lesssim \|f\|_{L^p([0,1);X)}.
\end{align*}
where the implicit constant depends only on $p,r,q$ and the space $X$.
\end{theorem}
So far the only Banach spaces which are known to have finite tile-type are exactly the complex interpolation spaces of Theorem \ref{t.main2}. From this 
point of view, the weaker Theorem \ref{t.main1} seems uninteresting compared to Theorem \ref{t.main2}. However, the formulation of Theorem \ref{t.main1}
only assumes finite tile-type. This provides an indication that the finite tile-type of a Banach space $X$ could be a strictly weaker 
hypothesis than that of $X$ being a complex interpolation space.

The Fourier scalar variational Carleson theorem is related to a number of topics, including 
 variational estimates for singular integrals \cite{JSW}; refined estimates in ergodic theory \cites{CLTT,MR2417419}; 
maximal inequalities \cite{1110.1070}, and approaches to extensions of results of Christ-Kiselev \cite{MR1952927} in spectral theory, see 
\cite{OSTTW}*{Appendix C}. Some of these continue to be under active development \cites{1110.1067,1203.5135}.  At some point, these topics might 
be ripe for investigation in the vector valued case.  

The rest of the paper is organized as follows. In section \ref{s.packets} we introduce tiles and trees in the time-frequency plane and define the
corresponding wave packets in terms of appropriate Walsh functions. In section \ref{s.tiletype} we review the definition of the tile-type of a Banach space, 
adjusted to the needs of the present paper. This definition is slightly weaker than the one in \cite{HL} but philosophically it is the 
same. We also discuss several structural properties of the trees that one needs to consider in the definition of the tile-type and show that tile-type 
$q$ implies martingale cotype $q$ for any Banach space $X$. We also recall the 
vector-valued version of L\'epingle's inequality which will play an important role later on in the proof. In section  \ref{s.linearization} we linearize the variational Carleson operator and introduce its variants $C_{r,\P}$ which characterize the tile-type of the Banach space $X$. We eventually linearize all our operators and reduce the main theorem to the statement of 
Propositions \ref{p.converse} and \ref{p.main}. In Sections \ref{s.tree} to \ref{s.summingup} we introduce all the necessary machinery from time-frequency analysis and prove Proposition \ref{p.main}. The last section \ref{s.inter} is devoted to the proof of Theorem \ref{t.main2}. The proof uses complex interpolation between the 
full range of $r$-variation inequalities, valid in any Hilbert space, and the $\infty$-variation bounds for intermediate spaces from \cite{HL}.

\section{Notation}\label{s.notation} Throughout the text $c,C$ will denote generic positive constants that might change even in the same line of text. We 
write $A\lesssim B$ if $A\leq c B$ for some numerical constant $c>0$ and $A\eqsim B$ if $A\lesssim B$ and $B\lesssim A$. In this paper the implicit constants that appear in various estimates may 
depend on the variation index $r$, the integrability index $p$, the tile-type index $q$ and the space $X$ itself, but we typically suppress this 
dependence since it is of no importance. We denote by $\N$ the set of non-negative integers and by $\RR$ the set of non-negative real numbers. Finally the 
dyadic intervals on the positive real line are denoted by $\mathcal D$. These are the intervals of the form $[n2^k,(n+1)2^k)$ where $n\in\N$ 
and $k\in\Z$. For any interval $\omega$ with endpoints $a<b$ we use the standard notation $\mathring{\omega}\coloneqq (a,b)$ for its interior as well as the notations 
$[\omega)\coloneqq [a,b)$ and $(\omega]\coloneqq (a,b]$. If no special notation is used by convention we use $\omega=[\omega)=[a,b)$.

For $k\in\Z$ we denote by $\Exp_k$ the conditional expectation with respect to dyadic intervals of length $2^k$:
\begin{align}\label{e.condk}
	\Exp_k f(x) \coloneqq\sum_{I\in\mathcal D, |I|=2^k} \frac{\ind_I (x)}{|I|} \int_I f(y)dy\eqqcolon\sum_{I\in\mathcal D, |I|=2^k} \Exp_If(x) ;
\end{align}
the dyadic maximal function is then $Mf(x)\coloneqq \sup_k |\Exp_kf(x)|$. Finally, we denote by $\bmo$ the dyadic $\bmo$ space on the positive real line, equipped with the norm
\begin{align*}
 \|f\|_{\bmo(\RR;X)}=\|f\|_\bmo \coloneqq \sup_{I\in\mathcal D} \frac{1}{|I|} \int_I \Abs{ f(x)-\langle f \rangle _I} dx,
\end{align*}
where $\langle f \rangle _I \coloneqq \frac{1}{|I|}\int_I f$. Note that, throughout the text, the notation $|\cdot|$ is used both for the absolute value as well as for the norm of the Banach space $X$, depending on context.

\section{Walsh wave packets, Tiles and Trees} \label{s.packets}

A \emph{tile} $P$ is a dyadic rectangle of area $1$ in the time-frequency plane $\RR\times \RR$, namely 
\begin{align}
P=I_P\times \omega_P= I_P\times \frac{1}{|I_P|}[n,n+1),\quad n\in\N,\ \ I_P\in\mathcal D,
\end{align}
If $P,P'$ are tiles we write $P\leq P'$ if $I_P\subset I_{P'}$ and $\omega_{P'} \subset \omega_{P}$. Likewise, a \emph{bitile} $P$ is a dyadic rectangle of area $2$:
\begin{align*}
P=I_P\times \omega_P=I_P\times \frac{2}{|I_P|}[n,n+1)=\bigcup_{v=0}  ^1  I_P\times \frac{1}{|I_P|}[2n+v,2n+v+1)\eqqcolon P_d \cup P_u.
\end{align*}
Thus each bitile $P$ has a `down-part' and an `up-part' which are also dyadic. Furthermore we write
\begin{align*}
	P\leq_u P' \iffdef P_u \leq P' _u \quad\mbox{and}\quad P\leq _d  P' \iffdef P_d\leq P'_d.
\end{align*}
The partial order for bitiles $P,P'$ is then defined as $P\leq P' \iffdef P\leq_u P' \text{ or } P\leq _d P'$. A tree $\T$ is a collection of bitiles $P$ for which there exist a top bitile $T$, which is not necessarily part of the collection, such that $P\leq T$ for all $P\in\T$. Note that in general a top of a tree is not uniquely defined. Similarly we say that $\T$ is an \emph{up-tree}  if $P\leq_u T$ for all $P\in\T$ and some top $T$ and a \emph{down-tree} if $P\leq _d T$ for all $P\in \T$. Trivially any tree can be decomposed into an up-tree and a down-tree:
\begin{align*}
\T=\T_u \cup \T_d,\quad \T_u\eqdef\{P\in\T: P\leq_u T\},\quad \T_d\eqdef \{P\in\T:P\leq_d\T\}.
\end{align*}

The Rademacher functions are defined as
\begin{align*}
r_i(x)\coloneqq \textnormal{sgn} \sin(2\pi \cdot 2^i x)=\sum_{k\in\N} \big( \ind_{2^{-i}[k,k+1/2)}(x)-\ind_{2^{-i}[k+1/2,k+1)}(x)    \big).
\end{align*}
If $n\in\N$ has binary expansion 
\begin{align*}
n=\sum_{i=0} ^\infty n_i 2^i,\quad n_i\in\{0,1\},
\end{align*}
we define the $n$-th Walsh function to be 
\begin{align*}
w_n(x)\coloneqq \prod_{i=0} ^\infty r_i(x)^{n_i}.
\end{align*}
Observe that we recover the Rademacher functions from the Walsh functions by means of $w_{2^i}(x)=r_i(x)$. It is well known and easy to see that the restrictions $\{w_n\ind_{[0,1)}\}_{n\in\N}$ form an orthonormal basis of $L^2(0,1)$.

With these definitions at hand we now associate to each tile $P\subset \RR\times \RR$  a \emph{wave packet} $w_P$ as follows. First we write the tile $P$ with respect to its time and frequency components:
\begin{align*}
P=I_P\times\omega_P = I_P\times \frac{1}{|I_P|} [n,n+1),\quad I_P\in\mathcal D, \quad n\in\N.
\end{align*}
The wave packet $w_P$ is now defined as
\begin{align*}
w_P(x)\coloneqq \frac{1}{|I_P|^\frac{1}{2}}\ind_{I_P}(x) w_n\big(\frac{x}{|I_P|}\big)\eqqcolon \frac{1}{|I_P|^\frac{1}{2}} w_P ^\infty (x).
\end{align*}
Observe that $w_P$ is $L^2$-normalized while $w_P ^\infty$ is $L^\infty$-normalized, hence the superscript $\infty$. The Haar functions are special cases of wave packets corresponding to the tiles of the form $P=I\times |I|^{-1}[1,2)$:
\begin{align*}
h_I(x)\coloneqq \frac{1}{|I|^\frac{1}{2}} \ind_I(x) r_0 \big(\frac{x}{|I|}\big)=w_{I\times|I|^{-1}[1,2)}(x).
\end{align*}
Given a bitile $P=P_d \cup P_u$ we use the notations $w_{P_u}$ and $w_{P_d}$ for the wave packets associated with the up-part and the down-part of the bitile $P$, respectively. 

If a collection of bitiles arises from a single up-tree then the following lemma gives a very useful description of the wave packets in terms of the simpler Haar functions. This lemma is taken from \cite{HL}*{Lemma 2.2} where we also refer the reader for the proof.
\begin{lemma}\label{l.haar}Let $\T$ be an up-tree with top $T$. For all $P\in \T$ we have
	\begin{align*}
		w_{P_d}(x)=\epsilon_{PT}\cdot w_{T_u} ^\infty(x) \cdot h_{I_P}(x),
	\end{align*}
where $w_{T_u} ^\infty$ is the $L^\infty$-normalized wave packet associated to the up-part of the top $T$ and $\epsilon_{PT}\in\{-1,+1\}$ is a constant factor that depends only on $P$ and $T$. In particular we have that
\begin{align*}
\ip f,w_{P_d},w_{P_d}=\ip f\cdot w_{T_u} ^\infty, h_{I_P},h_{I_P} \cdot w_{T_u} ^\infty.
\end{align*}
If $\T$ is a down-tree with top $T$ we have a symmetric statement:
	\begin{align*}
		w_{P_u}(x)=\tilde \epsilon_{PT}\cdot w_{T_d} ^\infty(x) \cdot h_{I_P}(x).
	\end{align*}
\end{lemma}

\section{Tile-type and cotype of a Banach space} \label{s.tiletype}
The notion of the tile-type of a Banach space $X$ was introduced in \cite{HL} to show that if a Banach space $X$ has finite tile-type then Carleson's theorem for Walsh-Fourier series is true for $X$-valued functions. More recently it was shown in \cite{HL1} that a  variant of the tile-type, adapted to the Fourier wave packets, also implies Carleson's theorem for the trigonometric system. Here we give a definition which is similar in spirit to that  in \cite{HL}*{\S3}. The main difference is that we only consider very special collections of trees in the definition of tile-type. These are essentially the trees generated by the selection algorithm in the size lemma, Lemma \ref{l.size}, and that lemma is the only place in the proof where the tile-type hypothesis is needed. In fact the reader is encouraged to briefly go through the statement and proof of the size lemma in order to gain some intuition on the definition that follows. The reason for giving this weaker but more complicated definition of tile-type is that it allows us to prove a partial converse of the variational Carleson theorem in Proposition \ref{p.converse}, namely that the $r$-variational boundedness of a Carleson-type operator implies that the space $X$ necessarily has tile-type $\tau$ for all $\tau>r$.

\subsection{Good collections of trees and tile-type of a Banach space}
We now describe the collections of trees that we want to consider in the definition of the tile-type. Let $\mathcal T=\{\T_j\}_j$ be a finite collection of up-trees, each consisting of finitely many bitiles, and set 
\begin{align*}
\P\coloneqq \{P:\ P\in\T_j\ \textrm{for some} \ j\}=\cup_j\T_j.
\end{align*}
Denote by $c(I)$ the center of some dyadic interval $I\in\mathcal D$. We will call the collection $\mathcal T$ \emph{u(p)-good} if it has the following property:

There is a reordering of the trees $\{\T_j\}_j$ and a choice of corresponding tops $\{T_j\}_j$ such that $\{c(\omega_{T_j})\}_j$ is an increasing sequence and
\begin{align*}
\T_j=\{P\in\P: P\leq_u T_j \ \textrm{and} \ P\nleq T_k\ \textrm {for all} \ k<j\}.
\end{align*}
There is a symmetric definition of a good collection of \emph{down-trees}, namely a collection $\mathcal T$ of down-trees will be called \emph{d(own)-good} if there is a reordering of the trees $\{\T_j\}_j$ and a choice of corresponding tops $\{T_j\}_j$ such that $\{c(\omega_{T_j})\}_j$ is a \emph{decreasing} sequence and
\begin{align*}
\T_j=\{P\in\P: P\leq_d T_j \ \textrm{and} \ P\nleq T_k\ \textrm {for all} \ k<j\}.
\end{align*}
We say that a Banach space $X$ has $u$-\emph{tile-type} $q$ if the estimate
\begin{align*}
	\bigg(\sum_{\T\in\mathcal T} \NOrm\sum_{P\in\T} \ip f,w_{P_d},w_{P_d}.L^q(\RR;X). ^q \bigg)^\frac{1}{q}\lesssim \|f\|_{L^q(\RR;X)},
\end{align*}
holds uniformly for all $u$-good collections $\mathcal T$. Similarly, we say that a Banach space $X$  has $d$-tile type $q$ if
\begin{align*}
	\bigg(\sum_{\T\in\mathcal T} \NOrm\sum_{P\in\T} \ip f,w_{P_u},w_{P_u}.L^q(\RR;X). ^q \bigg)^\frac{1}{q}\lesssim \|f\|_{L^q(\RR;X)},
\end{align*}
uniformly, for all $d$-good collections $\mathcal T$. It is actually not hard to see that the two definitions of tile-type, namely the one given with respect to $u$-good collections and the one given with respect to $d$-good collections, are equivalent. This is the content of the following lemma.

\begin{lemma}\label{l.updown} A Banach space $X$ has $u$-tile type $q$ if and only if it has $d$-tile type $q$.
\end{lemma}

\begin{proof}Let us assume that $X$ has $u$-tile type $q$ and let $\mathcal T$ be a $d$-good collection of down-trees. Let us fix a choice of tops $\{T_j\}_j$ ordered so that the sequence of centers $\{c(\omega_{T_j})\}_j$ is \emph{decreasing} and
\begin{align*}
\T_j=\{P\in\P: P\leq_d T_j \ \textrm{and} \ P\nleq T_k\ \textrm {for all} \ k<j\}.
\end{align*}
Let $\P$ be the collection of all tiles in $\mathcal T$ and define $2^N\coloneqq \sup_{P\in \P} \sup \omega_P$. For any bitile $P\in \P$, $P=I_P \times |I_P|^{-1} [n,n+2)$, we define the transformation:
\begin{align*}
 P\mapsto \tilde P \coloneqq I_P\times  [2^N - (n+2)|I_P|^{-1},2^N -n|I_P|^{-1} ) . 
\end{align*}
By the choice of $N$ the transformation above maps bitiles $P\subset \RR \times \RR$ into bitiles $\tilde P\subset \RR \times \RR$. Observe also that the down-part of a bitile $P$ is mapped to the up-part of the bitile $\tilde P$ and vice versa. Transforming the tops $\{ T_j \} _j$ accordingly we obtain a collection 
$\tilde {\mathcal T}$, consisting of up-trees, and a sequence of tops $\{ \tilde T _j \}_j$ which together form a $u$-good collection. By the assumption that $X$ has $u$-tile type $q$ we thus get
\begin{align*}
\bigg(\sum_{\tilde \T\in\tilde{ \mathcal  T}} \NOrm\sum_{\tilde P\in\tilde \T} \ip f,w_{\tilde P_d},w_{\tilde P_d}.L^q(\RR;X). ^q \bigg)^\frac{1}{q}\lesssim \|f\|_{L^q(\RR;X)}.
\end{align*}
Once this estimate is written down, the specific choice of tops $\{T_j\}_j$ is not relevant any more. For each tree $\T_j$ it is clear that we can choose a top $S_j$ with $|I_{S_j}|=\max_\ell |I_{T_\ell}|$. Thus the number $|I_{S_j}|$ does not depend on $j$. Applying Lemma \ref{l.haar} to the tiles $\tilde P$ belonging to the up-tree $\tilde \T $ with top $\tilde S $, we have
\begin{align*}
w_{\tilde P _ d}(x)= \tilde \epsilon _{PS} \cdot h_{I_ S}(x) \cdot w_{ {\tilde S} _u} ^\infty (x)\quad\mbox{and}\quad w_{ P _ u}(x)=   \epsilon _{PS} \cdot h_{I_ S}(x) \cdot w_{   S  _d} ^\infty (x).
\end{align*}
Define $2^m\coloneqq 2^N |I_S|$ and note that this number \emph{does not} depend on the specific choice of tree $\T$. In order to derive a relation between $ w_{ {\tilde S} _u} ^\infty$ and $w_{S  _d} ^\infty$ let us write $S=I_S \times |I_S|^{-1} [n,n+2)$ so that $S_d=I_S\times|I_S|^{-1}[n,n+1)$ and $\tilde S_u =I_S\times|I_S|^{-1} [2^m- n-1 ,2^m- n )$. We have
\begin{align*}
 w_{\tilde S _u} ^\infty (x) &= \ind_{I_S}(x) w_{2^m - n-1 } ({x}/ |I_S|  )=\ind_{I_S}(x) w_{2^m -1 } ({x}/ |I_S|  )w_n ({x}/ |I_S|  )\\
&=w_{2^m -1 } ({x}/ |I_S|  ) \cdot w_{S_d} ^\infty (x)\eqqcolon \phi_\P (x)\cdot w_{S_d} ^\infty (x),
\end{align*}
where $\phi_\P$ is a unimodular function that depends only on the collection $\P$. Thus
\begin{align*}
 \|f\|_{L^q(\RR;X)}\gtrsim \bigg(\sum_{\tilde \T\in\tilde{ \mathcal  T}} \NOrm\sum_{\tilde P\in\tilde \T} \ip f,w_{\tilde P_d},w_{\tilde P_d}.L^q(\RR;X). ^q \bigg)^\frac{1}{q}=\bigg(\sum_{  \T\in { \mathcal  T}} \NOrm\sum_{  P\in  \T} \ip f\phi_\P ,w_{  P_u},w_{  P_u}\phi_\P .L^q(\RR;X). ^q \bigg)^\frac{1}{q}.
\end{align*}
Replacing $f$ by $f \phi_\P ^{-1}$ in the previous estimate we conclude that $X$ has $d$-tile type $q$. The proof of the reverse implication is completely symmetric.
\end{proof}
\begin{remark} In view of Lemma \ref{l.updown} we will henceforth say that a Banach space $X$ has \emph{tile type} $q$ whenever it has $u$-tile type or $d$-tile-type $q$. We will also talk about good collections $\mathcal T$ without specifying whether we are talking about $u$-good or $d$-good collections. Furthermore, it is obvious that if a collection of trees can be split into a finite number $k$ of good collections then the tile-type inequality still holds for the original collection with some different constant depending on $k$. We will then say that $\mathcal T$ is a $k$-good collection, or just a good collection if it is clear that the number $k$ does not depend on anything interesting.
\end{remark}
The following lemma gathers some useful properties of good collections $\mathcal T$ and is the main ingredient in the proof of the partial converse in Proposition \ref{p.converse}.
\begin{lemma}\label{l.good} Let $\mathcal T=\{\T_j\}_{j=1} ^M$ be a good collection of up-trees and denote by $\P$ the set of all bitiles in $\mathcal T$. Let $\{T_j\}_{j=1} ^M$ be the collection of the corresponding tops from the definition a good collection, ordered so that $\{c(\omega_{T_j} )\}_j$ is increasing. We have the following properties:

\noindent{(i)} The down-parts of the bitiles in $\mathcal T$ are disjoint:
	\begin{align*}
		\textrm{if}\ \ P,P'\in\P \ \ \textrm{and}\ \ P\neq P' \ \ \textrm{then} \ \ P' _d \cap P_d =\emptyset.
	\end{align*}

\noindent{(ii)} Let $k(j,x) \coloneqq \max_k  \{1\leq \ k\leq j:  I_{T_k}\ni x  \}$ with the understanding that $\max \emptyset\coloneqq 0$. For $j\in\{1,2,\ldots,M\}$ define the measurable functions $N_j:\RR\to \RR$ as $N_j(x)\coloneqq c(\omega_{T_{k(j,x)}}) $. Let us also set $N_0(x)\equiv 0$ for convenience. For any fixed $x\in \RR$ the sequence $\{N_j(x)\}_j$ is increasing. Furthermore, for each $j$ and $x\in \RR$ we have
\begin{align*}
\{P\in\T_j:\ I_P\ni x\} =\{P\in\P: \ I_P\ni x,\ \ N_j(x)\in [\omega_{P_u}) \ \ \textrm{and} \ \ N_{j-1} (x) \notin \mathring{\omega}_{P} \}.
\end{align*}

\noindent{(iii)} For $1\leq r<\infty$ we have
\begin{align*}
	\sum_j \NOrm \sum_{P\in\T_j} \ip f,w_{P_d},w_{P_d}. L^r(\RR;X). ^r &=\int \sum_j \Abs{  \sum_{P\in\P} \ip f,w_{P_d},w_{P_d} (x)\ind_{\{ N_j (x) \in [\omega_{P_u}), \ N_{j-1}(x)\notin\mathring{\omega}_{P}   \}}}^r dx.
\end{align*}
\end{lemma}

\begin{proof} For (i) observe that if two bitiles $P,P'\in \P$ belong to the same tree $\T\in\mathcal T$ then we always have $P_d\cap P_d '=\emptyset$ since $\T$ is an up-tree. Suppose now that $P\in \T_j$, $P'\in\T_k $ where $\T_j,\T_k$ are two different trees in $\mathcal T$. Assume for the sake of contradiction that $P_d\cap {P' _d } \neq \emptyset$ so that $I_P\cap I_{P'}\neq \emptyset$ and $\omega_{P_d} \cap \omega_{P {'} _d}\neq \emptyset$. Then we have for example that $\omega_{P_d} \subseteq \omega_{P {'} _d}$. However, since $P,P'$ are different tiles we must actually have that $\omega_P\subseteq \omega_{P' _d}$. Thus 
	\begin{align*}
		\omega_{T_j} \subseteq \omega_P \subseteq \omega_{P' _d}\subset \omega_{P'}  \ \ \textrm{and} \ \ \emptyset \neq{I_{P}\cap I_{P'}}\subseteq I_{T_j}\cap I_{P'}
	\end{align*}
which implies that $P'\leq T_j$. We then get the following inequality for the centers of $\omega_{T_j}$, $\omega_{T_k}$:
	\begin{align*}
		c(\omega_{T_j})<\sup\omega_{T_j}\leq \sup \omega_{P} \leq \sup \omega_{P' _d} =\inf \omega_{P' _u}\leq \inf \omega_{T_{k,u}}=c(\omega_{T_k}).
	\end{align*}
By the definition of a good collection we thus have that  $j<k$ so that $P' \nleq T_j$, a contradiction.

We now prove (ii). First note that for every fixed $x\in\RR$ the sequence $\{N_j(x)\}_j$ is increasing as a composition of increasing functions of $j$. In order to prove the main claim in (ii)  we fix $x\in\RR$ and $j\in\{1\ldots,M\}$ and define the collections of bitiles
\begin{align*}
	S(j,x)\coloneqq \{P\in \T_j: I_P\ni x\}, \quad B(j,x)\coloneqq \{P\in\P: I_P\ni x, N_j(x)\in \omega_{P_u},\ \ N_{j-1}(x)\notin \mathring{\omega}_{P}\},
\end{align*}
where $N_j(x)$ is as in the statement of the lemma. We claim that $S(j,x)=B(j,x)$. If $x\notin I_{T_j}$ then both collections are empty: for $S(j,x)$ this is because $x\notin I_{T_j}\supseteq I_P$ while for $B(j,x)$ because $N_{j-1}(x)=N_j(x)$ in this case.

It remains to verify the claim when $x\in  I_{T_j}$ in which case $N_j(x)=c(\omega_{T_j})$ and $N_{j-1}(x)\in\{0,c(\omega_{T_{k(j-1,x)}})\}$.

Let $P\in S(j,x)$. Then $x\in I_P\subseteq I_{T_j}$ and by the definition of the good collection we have that $P\leq_u T_j$ and $P\nleq  T_k$ for any $k\leq j-1 $. The condition $P\leq_u T_j$ implies that $N_j(x)=c(\omega_{T_j})\in\omega_{P_u}$. If $x\notin\cup_{\ell\leq j-1}I_{T_\ell}$ then $N_{j-1}(x)=0$ which is never in the interior of any frequency interval thus $N_{j-1}(x)\notin \mathring{\omega}_{P}$ in this case. On the other hand if $x\in\cup_{\ell\leq j-1}I_{T_\ell}$ we have that $x\in I_{T_{k(j-1,x)}}$ and $P\nleq T_{k(j-1,x)}$ since $k(j-1,x)<j$. Since $x\in I_P \cap I_{T_{k(j-1,x)}}\neq \emptyset $ we must have $\omega_{T_{k(j-1,x)}}\nsubseteq \omega_P$ and thus $N_{j-1}(x)=c(\omega_{T_{k(j-1,x)}})\notin \mathring{\omega}_P$. This proves the inclusion $S(j,x)\subseteq B(j,x)$.

For the opposite inclusion assume that $P\in B(j,x)$. Since $x\in I_{T_j}$ we have that $N_j(x)=c(\omega_{T_j})\in\omega_{P_u}$ which is equivalent to $\omega_{T_{j,u}}\subseteq \omega_{P_u}$. Since $x\in I_P\cap I_{T_j}$ this shows that $P\leq_u T_j$. Now it is not hard to see that $P \nleq T_\ell $ whenever $\ell <j$. Indeed suppose that we had $P\leq T_\ell$ for some $\ell <j$. This would imply that $x\in I_P\subseteq I_{T_\ell}$ and thus $\ell \leq k(j-1,x)< j $. Furthermore we would have $\omega_{T_\ell},\omega_{T_j} \subseteq \omega_P$ so by the convexity of the interval $\omega_P$ and the fact that the sequence $\{c(\omega_{T_j})\}_j$ is increasing we would conclude that $N_{j-1}(x)=c(\omega_{T_{k(j-1,x)}})\in \mathring{\omega}_P$, contradicting the second condition in the definition of $B(j,x)$. This proves the inclusion $B(j,x)\subseteq S(j,x)$ and thus concludes the proof of (ii).

Finally part (iii) of the lemma is an obvious application of the identity $S(j,x)=B(j,x)$.
\end{proof}
In the following paragraphs of this section we will investigate how the tile-type condition relates to the classical cotype of a Banach space. For this we will need to be able to view the dyadic intervals inside $[0,1)$ as the time intervals of bitiles of a suitable good collection. This is the content of the following lemma.
\begin{lemma}\label{l.converse} Let $\mathcal J$ be a finite collection of dyadic intervals in $[0,1)$ and define $2^{-N}\coloneqq \min_{I\in\mathcal J} |I|$.

\noindent{(i)} Let $\T\coloneqq \{I\times [2^{N+1}-2|I|^{-1},2^{N+1}):\ I\in \mathcal J\}$. Then $\T$ is an up-tree and $\mathcal J=\{I_P: P=I_P\times \omega_P \in \T\}$.

\noindent{(ii)} There exists a good collection of up-trees $\mathcal T=\{\T_j\}_{j=0} ^N$ such that $\T=\cup_j \T_j$ and 
\begin{align*}
\{I\in \mathcal J: |I|=2^{j-N}\} = \{I_P: P=I_P\times \omega_P \in \T_j\},\quad 0\leq j \leq N.
\end{align*}

\noindent{(iii)} For $1\leq r <\infty$ and every $f\in L^r([0,1);X)$ we have the identity
\begin{align*}
	\sum_{j=0} ^N  \Abs{ \sum_{P\in \T_j} \ip w_{\mathcal J} ^\infty f,w_{P_d},w_{P_d}(x)}^r=\sum_{j=0} ^N  \Abs{ \sum_{\substack{I\in\mathcal J \\ |I|=2^{-j}}} \ip  f,h_I,h_I(x)}^r,
\end{align*}
where $w_\mathcal J ^\infty$ is a unimodular function that depends only on the collection $\mathcal J$.
\end{lemma}

\begin{proof} For (i) it is enough to notice that the tile $T\coloneqq [0,1)\times [2^{N+1}-2,2^{N+1})$ satisfies $P\leq_u T$ for all $P\in\mathbb T$. Let us now show (ii). For $I\in \mathcal J$ we set $\omega_I\coloneqq [2^{N+1}-2|I|^{-1}, 2^{N+1})$ and for all $ 0\leq j \leq N$ we define the trees
\begin{align*}
\T_j\coloneqq\{ P=I\times \omega_I: \ I\in\mathcal J,\ |I|=2^{j-N }\}.
\end{align*}
We define an appropriate top $T_j$ for each tree $\T_j$ by setting
	\begin{align*}
	T_j\coloneqq [0,2) \times  [2^{N+1}-2^{N-j},2^{N+1}-2^{N-j}+1).
	\end{align*}
Suppose that $P=I_P\times \omega_P\in \T_j$ for some $j\in\{0,1,\ldots,N\}$. Then
	\begin{align*}
		\omega_{T_j}=[2^{N+1}-|I_P|^{-1},2^{N+1}-|I_P|^{-1}+1)\subset \omega_{P_u}
	\end{align*}
and obviously we always have that $I_P\subset I_{T_k}$. Thus each $T_j$ is a top of $\T_j$ and hence each $\T_j$ is an up-tree. By construction the sequence $\{c(\omega_{T_j})\}_{j\leq N} $ is strictly increasing and furthermore the intervals $\omega_{T_j}$ are disjoint. We first show that the collection $\{\T_j\}_{j\leq N}$ satisfies
	\begin{align*}
	 \T_j=\{P\in\P: P\leq _u  T_j \ \textrm{and}\  P\nleq_u T_k\ \textrm{for all}\  k<j\},
	\end{align*}
where $\P$ is the collection of all the bitiles contained in the trees $\T_j$. Let $P\in\T_j$. We already saw that $P\leq_u T_j$. Furthermore for $k<j$ we have that
$\sup \omega_{T_k}\leq \inf \omega_{T_j}$. Observe however that $\inf {\omega_{T_j}}=2^{N+1}-2^{N-j}=2^{N+1}-|I_P|^{-1}=\omega_{P_u}$ since $P\in \T_j$. Thus $\sup \omega_{T_k}\leq \inf \omega_{P_u}$ which implies that $P\nleq_u T_k$ whenever $k<j$. This proves
\begin{align*}
 \T_j\subseteq \{P\in\P: P\leq _u  T_j \ \textrm{and}\  P\nleq_u T_k\ \textrm{for all}\  k<j\}.
\end{align*}
Now assume that $P\in \{P\in\P: P\leq _u  T_j \ \textrm{and}\  P\nleq_u T_k\ \textrm{for all}\  k<j\}$. Then we have $P\leq_u T_j$ thus $\omega_{T_{j,u}}\subseteq \omega_{P_u}$ which implies that $2^{N+1}-|I_P|^{-1}\leq 2^{N+1}-2^{N-j}\Leftrightarrow |I_P| \leq 2^{j-N}$. We claim that in fact $|I_P|=2^{j-N}$. Indeed, if $|I_P|\leq 2^{(j-1)-N}$ then we would get that $\omega_{T_{j-1,u}}\subseteq \omega_{P_u}$ and this in turn would give that $P\leq_u T_{j-1}$ which is a contradiction. Since $P\in\P$ and $|I_P|=2^{j-N}$ we get that $P\in \T_j$.

Observe that if $P\in \T_j$ for some $j$  then by construction $P\cap T_k=\emptyset$ for all $k<j-1$.  Thus we have $P\nleq_u T_k\Leftrightarrow P\nleq T_k$ for $k<j-1$. 
We now split the collection $\mathcal T$ into two collections by setting say $\mathcal T_1=\{\T_{2j}\}_j$ and $\mathcal T_2=\{T_{2j+1}\}_j$ and each collection $\mathcal T_\nu$, $\nu=1,2$, is good. 
This shows that the original collection $\mathcal T$ is a $2$-good collection.

For (iii), remember that the trees $\T_j$ share a common top $T=[0,1)\times [2^{N+1}-2,2^{N+1})$. Thus, Lemma \ref{l.haar} implies that
\begin{align*}
\Abs{	\sum_{P\in \T_j}\ip f,w_{P_d},w_{P_d}(x)}&=\Abs{\sum_{P\in\T_j} \ip f w_{T_u} ^\infty, h_{I_P},w_{T_u} ^\infty h_{I_P}}= \Abs{\sum_{I\in \mathcal J} \ip f w_{T_u} ^\infty, h_{I_P}, h_{I_P}},
\end{align*}
which proves the claim in (iii) by setting $w_\mathcal J ^\infty\coloneqq w_{T_u} ^\infty$ and replacing $f$ by $f w_\mathcal J ^\infty$.
\end{proof}
 Finally we recall the main result proved in \cite{HL} concerning the tile-type of an interpolation space $X$. Observe that by Lemma \ref{l.good} the down-parts of all bitiles in a good collection $\mathcal T$ are disjoint; thus the following Proposition is identical to \cite{HL}*{Proposition 3.1}.
\begin{proposition}\label{p.tiletype} A necessary condition for tile-type $q$ is that $X$ is a UMD space
and $q\geq 2$. If a UMD space has tile-type $q$, it has tile-type $p$ for all $p \in [q,\infty)$. 
Every Hilbert space has tile-type $2$, and every complex interpolation space $[Y,H]_\theta$, $\theta\in(0,1)$, between a UMD space $Y$ and a Hilbert space $H$ has tile-type $2/\theta$.
\end{proposition}

\subsection{Tile-type implies cotype} We observe in this paragraph that the hypothesis that a Banach space $X$ has tile-type $q$ implies that the space $X$ has Rademacher and martingale cotype equal to $q$. We recall the relevant definitions. 

Let $2\leq q \leq \infty$. We say that a Banach space $X$ has (Rademacher) cotype $q$ if
\begin{align*}
\big(\sum_{j\geq 0} |x_j|^q \big)^\frac{1}{q}\lesssim \Norm \sum_{j\geq 0}  r_j   x_j .L^q([0,1);X).
\end{align*}
holds uniformly for all finite sequences $\{x_j\}_j\subset X$, where $\{r_j\}_j$ are the Rademacher functions on $[0,1)$.

On the other hand, we say that $X$ has martingale cotype $q\in[2,\infty]$ (or $M$-cotype $q$) if for all $X$-valued martingales $\{M_n\}_n$ we have
\begin{align*}
\big( \Exp \sum_{n\geq 0} |M_n -M_{n-1}|^q \big)^\frac{1}{q}\lesssim \big(\sup_n \Exp |M_n|^q\big)^\frac{1}{q}.
\end{align*}
Every Banach space trivially has cotype and $M$-cotype $\infty$. 
In general the notion of $M$-cotype is stronger than that of usual cotype but the two notions are equivalent in the case that $X$ has the UMD property. 
Finally we note that martingale cotype is equivalent to Haar cotype, meaning that it is suffices to consider Haar martingales in the definition of $M$-cotype. See \cite{PIS}.

The following proposition shows that tile-type implies $M$-cotype with the same index :
\begin{proposition}\label{p.cotype} Suppose that the Banach space $X$ has tile-type $q\geq 2$. Then $X$ is UMD and has $M$-cotype $q$.
\end{proposition}
\begin{proof} The fact that tile-type $q$ implies the UMD property is already contained in Proposition \ref{p.tiletype} but we include a proof here for the sake of completeness. It will suffice to show that
\begin{align*}
\Norm \sum_{I\in \mathcal J} \epsilon_I \ip f , h_I , h_I . L^r ( [0,1);X) . \lesssim \|f\|_{L^r([0,1);X)}. 
\end{align*}
where $\mathcal J$ is any finite collection of dyadic intervals inside $[0,1)$,  $f\in L^r([0,1);X)$, $\epsilon_I\in\{-1,+1\}$ and $r$ is some fixed exponent in $(1,\infty)$. Because of the following trivial estimate
\begin{align*}
\ABs{ \sum_{I\in \mathcal J} \epsilon_I \ip f , h_I , h_I}&=  \ABs{\sum_{\substack{I\in \mathcal J\\ \epsilon_I=1}}  \ip f , h_I , h_I -\sum_{\substack{ I\in \mathcal J \\ \epsilon_I=-1}}  \ip f , h_I , h_I }
\leq \ABs{\sum_{\substack{I\in \mathcal J\\ \epsilon_I=1}}  \ip f , h_I , h_I} +\ABs{\sum_{\substack{I\in \mathcal J\\ \epsilon_I=-1}}   \ip f , h_I , h_I},
\end{align*}
it will actually suffice to prove that
\begin{align*}
\Norm \sum_{I\in \mathcal J} \epsilon_I \ip f , h_I , h_I . L^r ( [0,1);X) . \lesssim \|f\|_{L^r([0,1);X)}. 
\end{align*}
whenever $\epsilon_I\in \{0,1\}$. However this amounts to showing that
\begin{align*}
\Norm \sum_{I\in \mathcal J'}   \ip f , h_I , h_I . L^r ( [0,1);X) . \lesssim \|f\|_{L^r([0,1);X)},
\end{align*}
for any finite collection $\mathcal J'$ of dyadic intervals in $[0,1)$. Consider the up-tree given by (i) of Lemma \ref{l.converse} applied to the collection $\mathcal J'$:
\begin{align*}
\T\coloneqq\{I \times [2^{N+1}-2|I|^{-1}, 2^{N+1} ):\ I\in\mathcal J'\},
\end{align*}
where $N$ is such that $|I|\geq 2^{-N}$ for all $I\in\mathcal J'$. Setting $gw_{T_u} ^\infty\coloneqq f$ we use Lemma \ref{l.haar} to write
\begin{align*}
 \Norm \sum_{I\in \mathcal J'}   \ip f , h_I , h_I . L^q ( [0,1);X) . ^q&=\int\Abs{\sum_{P\in \T} \ip f ,h_{I_P}, h_{I_P}(x) }^qdx
\\
&=\int \Abs{\sum_{P\in \T} \ip g,w_{P_d}, w_{P_d}(x) }^qdx \lesssim \|f\|_{ L^q ( [0,1);X) } ^q
\end{align*}
where in the last inequality we used the tile-type hypothesis for the collection consisting of the single tree $\T$. This however is the UMD condition for Haar martingales with $r=q$.

We will now show that $X$ has martingale cotype $q$. By Lemma \ref{l.converse} we have for every positive integer $N$ that
	\begin{align*}
		\sum_{j=0} ^N \Abs{ \sum_{|I|=2^{-k}} \ip  f,h_I,h_I(x)}^q=	\sum_{j=0} ^N \Abs{ \sum_{P\in \T_j} \ip w_N ^\infty f,w_{P_d},w_{P_d}(x)}^q,
	\end{align*}
where $\{\T_j\}_{j=0} ^N$ is a good collection of up-trees. Since $X$ has tile-type $q$ the right hand side is controlled by $\|f\|_{L^q([0,1);X)}$. Thus
\begin{align*}
\sum_{0\leq k\leq N }\Norm \sum_{|I|=2^{-k} } \ip f,h_{I_P},h_{I_P}.L^q([0,1);X).  ^q \lesssim \|f\|_{L^q([0,1);X)} ^q,
\end{align*}
with the implicit constant not depending on $N$. This is the cotype condition for Haar martingale differences which by \cite{PIS} is equivalent to $X$ having  martingale cotype $q$.
\end{proof}

\subsection{Vector-valued L\'epingle inequality}
The variational Carleson theorem, in the scalar case, depends on certain jump inequalities originally due to L\'epingle \cite{Lep}. This fact has been recorded and well understood in several papers as for example in  \cite{DOLA}, \cite{OSTTW}, \cite{JSW}. For the Banach space case that we are considering we will need the appropriate vector-valued extension proved by Pisier and Xu in \cite{PX}:
\begin{theorem}[\cite{PX}*{Theorem 4.3}]\label{p.leptile} Suppose that $X$ has cotype $\tau$ for all $\tau > q$. Then we have L\'epingle's inequality for functions $f\in L^p(\RR;X)$:
\begin{align*}
\norm \Exp_n f  . L^p(\RR;\mathcal V^r(X)). \lesssim \|f\|_{L^p(\RR;X)},
\end{align*}
for all $r>q$ and $1<p<\infty$. 
\end{theorem}
Here we remember that $\Exp_n$ is the conditional expectation with respect to dyadic intervals of length $2^n$, as defined in \eqref{e.condk}. By Proposition \ref{p.cotype} one can replace the cotype $\tau>q$ hypothesis in Theorem \ref{p.leptile} by the hypothesis that $X$ has tile-type $\tau$ for all $\tau> q$. We will use this fact in what follows without further comment.

\section{Linearization of the Variational Carleson operator} \label{s.linearization}
In this section we linearize the variational norm of the partial Walsh-Fourier sums of a function $f$, using more or less standard arguments as in \cite{T}, \cite{HL}, \cite{DOLA} and \cite{OSTTW}. We reduce the statement of Theorem \ref{t.main1} to an analogous statement about some closely related linearized versions of the variational Carleson operator which are more amenable to the time-frequency analysis techniques and interpolation. We carry the tile-type hypothesis throughout the section in the statements of our reduced theorems.

For any collection of bitiles $\P$ we define the operator  
\begin{align}\label{e.crp}
C_{r,\P}f(x)\coloneqq \sup_{K,N_0<\cdots<N_K}\bigg(\sum_{j=1} ^K\Abs {\sum_{P\in\P} \ip f,w_{P_d},w_{P_d}(x)\ind_{\{ N_j  \in [\omega_{P_u} ), \ N_{j-1}\notin\mathring{\omega}_P  \}}}^r \bigg)^\frac{1}{r},
\end{align}
where the supremum is taken over all positive integers $K$ and all  \emph{non-negative real numbers} $N_0< N_1< \cdots < N_K$. 
There is a symmetric version, denoted by $ \tilde C_{r,\P}f(x)$,  in which the down-tiles are replaced by up-tiles and the condition in the indicator is replaced by 
$ N_j \notin\mathring{\omega}_P\,,\ N _{j-1} \in (\omega _{P_d} ] $, namely:
\begin{align*}
\tilde C_{r,\P}f(x)\coloneqq \sup_{K,N_0<\cdots<N_K}\bigg(\sum_{j=1} ^K\Abs {\sum_{P\in\P} \ip f,w_{P_u},w_{P_u}(x)\ind_{\{ N_j  \in (\omega_{P_d} ], \ N_{j-1}\notin\mathring{\omega}_P  \}}}^r \bigg)^\frac{1}{r},
\end{align*}
These operators are formed over all bitiles, namely, one only requires $ I_P\subset [0, \infty )$. 

The second statement in Theorem~\ref{t.main_new} will be a consequence of the following theorem for $C_{r,\P}$ and its symmetric analog for $\tilde C_{r,\P}$:
\begin{theorem}\label{t.secondary} Let $X$ be a Banach space with tile-type $\tau$ for all $\tau>q$ and $\P$ be any collection of bitiles. We have
\begin{align}
\norm C_{r,\P} f. L^p(\RR;X). \lesssim_{p,r,q} \norm f.L^p(\RR;X).,
\end{align}
whenever $q<r< \infty$ and $0<\frac{1}{p}<\frac{1}{r'}-\frac{q-2}{r}$.
\end{theorem}

Concerning the proof of this Theorem, we focus on the operator $ C_{r,\P}f(x)$, using 
in particular the partial order on bitiles and their organization into trees, among other techniques. The main hypothesis is that the space $X$ has finite tile type arbitrarily close to some number $q$. The reader here should prefer to think of the tile type hypothesis in the formulation given for families of up-trees, that is, in the equivalent formulation of the $u$-tile type. For the operator, $ \tilde C_{r,\P}f$, the proof is completely symmetric, in view of Lemma \ref{l.updown}, where the role of the down tile $ P_d$ is analogous 
to that of the up-tile, and vice-versa. For all the considerations concerning the operator $ \tilde C_{r,\P}f$ we switch our point of view to the formulation of the $d$-tile type. Bearing this in mind it is routine to adjust the arguments in this paper, given for the operator $  C_{r,\P}f$, in order to give the corresponding proof for the symmetric operator $ \tilde C_{r,\P}f$. We thus omit any further discussion concerning the proof of Theorem \ref{t.secondary} for the dual operator $ \tilde C_{r,\P}f$.

We briefly describe how to conclude Theorem \ref{t.main1} from Theorem \ref{t.secondary}:
\begin{proof}[Proof of Theorem~\ref{t.main1}.]  
For integers $ 0<\zeta < \zeta '$, let $ \Omega  _{\zeta , \zeta '}$ be the maximal dyadic intervals 
$ \omega  \subset [ \zeta , \zeta ')$.  These intervals partition $[ \zeta ,\zeta ')$, and moreover we have 
\begin{equation*}
S _{\zeta '} f - S _{\zeta } f = \sum_{\substack{\textup{$ P$ is a tile}\\I_P\subset [0,1) ,\  \omega _P \in \Omega  _{\zeta ,\zeta '} }} \langle f, w_P\rangle w_{P} \,. 
\end{equation*}  
This follows from \cite{T}*{Corollary 8.3} and is a variant of the formula  \cite{T}*{p.~68-69}. Now,  let $ \Omega ^{u/d} _{\zeta , \zeta '} $ be those intervals $ \omega \in \Omega _{\zeta , \zeta '}$ for which 
$ \omega $ is the up/down--half of its parent.
Let $ \mathbf P ^{ u/d} _{\zeta , \zeta '}$ be the collection of \emph{bitiles} such that $ \omega _{P _{u/d}} \in \Omega ^{u/d} _{\zeta , \zeta '}$, 
and $ I _{P}\subset [0,1)$.  
We then have 
\begin{equation*}
S _{\zeta '} f - S _{\zeta } f  
= 
\sum_{\sigma \in \{u,d\}} 
\sum_{P\in \mathbf P ^{\sigma } _{\zeta , \zeta '}}  \langle f, w_{P_\sigma}\rangle w_{P_\sigma}  \,.  
\end{equation*}

We have $ P \in \mathbf P ^{d} _{\zeta ,\zeta '} $ if and only if $ I_P \subset [0,1)$, 
$ \zeta \notin\mathring\omega _{P}$, and $ \zeta '\in [\omega _{P_u})$, conditions 
in agreement with the conditions on $ N _{j-1}, N_j$ in the definition of $ C_{r,\P}f$.  
In the symmetric case, we have 
$ P \in \mathbf P ^{u} _{\zeta ,\zeta '} $ if and only if $ I_P \subset [0,1)$, 
$ \zeta  \in (\omega _{P_d} ] $, and $ \zeta ' \notin\mathring \omega _{P}$.  
All together, for any $ K$, $ N_0 < \cdots < N_K$, we have 
\begin{align*}
\sum_{j=1} ^{K} \lvert  S _{N _{j-1} } f - S _{N _{j}}\rvert ^{r} 
& = 
\sum_{j=1} ^{K} 
\Bigl\lvert 
\sum_{\sigma \in \{u,d\}} 
\sum_{P\in \mathbf P ^{\sigma } _{N _{j-1} , N_j}}  \langle f, w_{P_\sigma}\rangle w_{P_\sigma} 
\Bigr\rvert ^{r} 
\\
& \lesssim  \sum_{\sigma \in \{u,d\}} 
\sum_{j=1} ^{K} 
\Bigl\lvert 
\sum_{P\in \mathbf P ^{\sigma } _{N _{j-1} , N_j}}  \langle f, w_{P_\sigma}\rangle w_{P_\sigma} 
\Bigr\rvert ^{r} 
\\
& \lesssim (C_{r,\P_1}f) ^{r} + (\tilde C_{r,\P_2}f) ^{r} \,,
\end{align*}
for some fixed collections of bitiles $\P_1$ and $\P_2$. Using Theorem \ref{t.secondary}, which is valid for arbitrary collections $\P$, and the pointwise inequality just proved, completes the proof. 
\end{proof}

The first statement in Theorem~\ref{t.main_new} is the content of:
  \begin{proposition}\label{p.converse} Let $X$ be some Banach space and suppose that for any collection of bitiles $\P$, the operator $C_{r,\P} $, or the operator $\tilde C_{r,\P}$, satisfies the conclusion of Theorem \ref{t.secondary} with $p=r$:
\begin{align}
 \norm C_{r,\P} f. L^r(\RR;X). \lesssim \norm f.L^r(\RR;X). ,
\end{align}
whenever $q<r< \infty$. Then $X$ has tile-type $\tau$ for all $\tau>q$ and, a fortiori, $X$ has cotype $\tau$ for all $\tau>q$.
\end{proposition}
\begin{proof} We will prove the proposition assuming that the operator $C_{r,\P}$ is bounded on $L^r(\RR;X)$ for all $r>q$. Let $\mathcal {T}=\{\T_j\}$ be a u-good collection of up-trees. By Lemma \ref{l.good}, (ii) and (iii), there is an increasing sequence of integer valued functions $\{N_j(x)\}_j$, such that
\begin{align*}
\sum_j \NOrm \sum_{P\in\T_j } \ip f,w_{P_d},w_{P_d}  .  L^r(\RR;X) . ^r &= \sum_j \NOrm \sum_{P\in\P } \ip f,w_{P_d},w_{P_d}  \ind_{\{ N_j  \in [\omega_{P_u}), \ N_{j-1}\notin\mathring{\omega}_{P}  \}  }  .  L^r(\RR;X) . ^r 
\\
&\leq \norm C_{r,\P}f .L^r(\RR;X). ^r\lesssim_r \|f\|_{L^r(\RR;X)} ^r,
\end{align*}
since $C_{r,\P}$ is bounded on $L^r(\RR;X)$. If the hypothesis is true for $\tilde C_{r,\P}$ we consider $d$-good collections of trees and show the corresponding statement for the $d$-tile type. The conclusion then follows	by using the analogue of Lemma \ref{l.good} for $d$-good collections.
\end{proof}

Following \cite{DOLA} we consider the linearized version of $C_{r,\P}$ given by
\begin{align*}
C_\P f(x)= C_{r,a,\P}f(x)& \coloneqq \sum_{j=1} ^{K(x)} \sum_{P\in\P} \ip f ,w_{P_d},w_{P_d}(x) \ind_{\{ N_j (x) \in [\omega_{P_u}), \ N_{j-1}(x)\notin\mathring \omega_P  \}}a_j(x),
\end{align*}
where $K,N_1,\ldots N_K:\RR\to \R_+$ are arbitrary measurable functions and $a=\{a_j\}_j$ is a sequence of $X^*$-valued functions with $\sum_{j=1} ^{K(x)} |a_j(x)|^{r'}=1$. The expression for the operator $C_\P$ can be simplified by writing
\begin{align*}
C_\P f(x)&=\sum_{P\in\P} \ip f ,w_{P_d},w_{P_d}(x) \sum_{j=1} ^{K(x)} \ind_{\{ N_j (x) \in [\omega_{P_u}), \ N_{j-1}(x)\notin\mathring \omega_P  \}}a_j(x)
=\sum_{P\in\P} \ip f ,w_{P_d},w_{P_d}(x) \  a_{P}(x),
\end{align*}
where 
\begin{align*}
a_P(x)\coloneqq \sum_{j=1} ^{K(x)} \ind_{\{ N_j (x) \in [\omega_{P_u}), \ N_{j-1}(x)\notin\mathring \omega_P \}}a_j(x).
\end{align*}
The operator $C_\P $ depends on both $r$ and the choice of the sequence $a$ but we suppress this fact in what follows in order to simplify our notation. 

Note here that our assumption that $X$ has tile-type $\tau\in(q,\infty)$ can be replaced by the assumption that $X$ has tile-type exactly $q$. This is because all our conclusions are given in terms of \emph{open} intervals with respect to $p,r$ and $q$. Via a standard restricted weak-type interpolation argument, as for example in \cite{T}*{Chapter 3}, the proof of Theorem \ref{t.secondary} reduces to the proof of the following statement:
\begin{proposition}\label{p.main} Suppose that $X$ is a Banach space with tile-type $q\geq 2$. Let $F,E\subset \RR$ be measurable sets with $|F|,|E|<+\infty$. Then there are major subsets $ E'\subseteq E$ and $ F'\subseteq F $ with  either $ E'=E$ or $ F'=F$, 
such that, for all $f:X\to \RR$ with $|f|\leq \ind_{F'}$, and all $g:X^*\to\RR$ with $|g|\leq \ind_{ E'}$, we have
\begin{align*}
|\ip C_\P f,g, |\lesssim |F|^\frac{1}{p} |E|^\frac{1}{p'},
\end{align*}
whenever $	\max(q,p'(q-1))<r<\infty$.
\end{proposition}
Here we say that $ E'\subset E$ is a \emph{major subset of $ E$ } if $ | E'|\geq \frac{1}{2}|E|$. 

\begin{remark} Observe that
\begin{align*}
|\ip C_\P f,g, |\leq \sum_{P\in\P} \abs{ \ip f ,w_{P_d},\ip w_{P_d} a_{P},g, }= \sum_{P\in\P} \epsilon_P \ip f ,w_{P_d},\ip w_{P_d} a_{P},g, = |\ip  C_\P ^+ f,g,|
\end{align*}
for some choice of signs $\epsilon_P\in\{-1,+1\}$, where
\begin{align*}
 C_\P ^+ f (x)\coloneqq \sum_{P\in\P} \epsilon_P \ip f ,w_{P_d},w_{P_d}(x) \  a_{P}(x).
\end{align*}
We will thus prove the estimate in Proposition \ref{p.main} for the larger operator $ C_\P ^+$, which we immediately rename again to $C_\P $, and any \emph{finite} collection of bitiles $\P$.
\end{remark}

\section{The tree lemma} \label{s.tree} 
Let $\P$ be a finite collection of bitiles. The \emph{density} of the collection $\P$ is
\begin{align*}
\density(\P)\coloneqq \sup_{P\in\P} \sup_{P'\geq P} \bigg(\frac{1}{|I_{P'}|}\int _{I_{P'}} |g(x)|^{r'} \sum_{j:\ N_j(x)\in \omega_{P'}}|a_j(x)|^{r'}dx  \bigg)^\frac{1}{r'},
\end{align*}
where we remember that $N_j:\R_+\to \R_+$ are measurable functions, $q\geq 2$ is the tile-type of the Banach space $X$ and $r> q$.
We define the \emph{size} of a collection $\P$ to be
\begin{align*}
\size(\P)\coloneqq \sup_{\T\subseteq \P\ \mbox{\tiny up-tree}} \bigg(\frac{1}{|I_T|} \int \Abs{\sum_{P\in\T}\ip f,w_{P_d},w_{P_d} (x)}^q dx\bigg)^\frac{1}{q}.	
\end{align*}	
\begin{lemma}[Tree lemma]\label{l.tree} For every tree $\T$ we have
\begin{align*}
	\norm g C_\T f.L^s(\RR). = \Norm \sum_{P\in\T} \ip f,w_{P_d},  w_{P_d}a_P g .L^s(\RR). \lesssim \size(\T) \density(\T) |I_T|^\frac{1}{s},
\end{align*}
for all $1\leq s \leq r'$.
\end{lemma}

We will prove the lemma for the case $s=r'$ which, by H\"older's inequality,  implies the conclusion for $1\leq s\leq r'$ as well. Let $\mathcal J$ be the collection of maximal dyadic intervals contained in $I_T$ that do not contain any $I_P$, $P\in\T$. The intervals in the collection $\mathcal J$ form a partition of $I_T$ thus 
\begin{align*}
\norm g C_\T f.L^{r'}	(\RR). & = \big( \sum_{J\in\mathcal J} \int_J \Abs{\sum_{\substack{ P\in \T\\ I_P\supsetneq J}} \epsilon_P\ip f,w_{P_d},w_{P_d}(x)a_P(x)g(x) }^{r'} dx\big)^\frac{1}{r'}
\\
&= \bigg( \sum_{J\in \mathcal J}\Norm  \sum_{\substack{ P\in \T \\ I_P\supsetneq J}}\epsilon_P \ip f,w_{P_d},w_{P_d} a_P g . L^{r'}(J). ^{r'} \bigg)^\frac{1}{r'}.
\end{align*}
We set for $P\in\T$ and $j\geq 1$ 
\begin{align*}
A(P,j)\coloneqq I_P\cap \{x:\ N_{j-1}(x)\notin \mathring \omega_P, \ N_j(x)\in [\omega_{P_u}) \}.
\end{align*}

We gather some  auxiliary calculations in the following lemma:
\begin{lemma}\label{l.aux} Fix a tree $\T$ and a top $T$ of $\,\T$ and consider the partition of $I_T$ into the intervals $J\in\mathcal J$. Let $J\in\mathcal J$ and denote by $J^{(1)}$ the dyadic parent of $J$. There exist bitiles $Q(J)\in\T$ and $P(J)=J^{(1)}\times \omega(J)$ such that:

\noindent{(i)} $Q(J)\leq P(J) \leq T$.

\noindent{(ii)} For every $j\in[1,K(x)]$ we have the pointwise inequality: $ \ind _J \ind_{A(P,j)} \leq \ind_{\{x:\ N_j(x)\in\omega(J)\}}$.

\noindent{(iii)} We have the estimate
\begin{align*}
\int_J |g(x)|^{r'} \sum_{j:\ N_j(x)\in \omega(J)} |a_j(x)|^{r'}dx\lesssim |J|\density(\T)^{r'}.
\end{align*}
\end{lemma}

\begin{proof} Fix some $J\in\mathcal J$. Since $J$ is maximal with the property that it doesn't contain any $I_P$, $P\in\T$, there is some bitile $Q(J)\in\T$ such that $I_{Q(J)}\subseteq J^{(1)}$, where $J^{(1)}$ is the dyadic parent of $J$. Observe that we must have $J^{(1)} \subseteq I_T$. Define the frequency interval $\omega(J)$ with $| \omega(J)|=2/|J^{(1)}|$ and such that $\omega_T \subseteq   \omega(J) \subseteq \omega_{Q(J)}$. Thus the bitile $  P(J) \coloneqq J^{(1)}\times \omega (J)$ satisfies $Q(J)\leq  P(J) \leq T$. This proves \textit{(i)}.

Now for all bitiles $P\in\T$ such that $I_P\cap J\neq \emptyset$ we have $J\subsetneq I_P$ by the maximality of $J$ which implies that $|I_P|\geq |J^{(1)}|$.
For every such $P$ we thus have $\emptyset\neq \omega_T\subseteq \omega_P\cap  \omega(J)$ and so $\omega_P\subseteq  \omega(J)$. We conclude
\begin{align*}
\bigcup_{P\in\T: \ I_P\cap J\neq\emptyset } \omega_{P} \subseteq  \omega(J).
\end{align*}
Let $x$ be such that $\ind_J(x) \ind_{A(P,j)}(x)\neq 0.$ Then  $I_P\cap J\neq \emptyset$ and $N_j(x)\in \omega_{P_u}\subset \omega_P\subseteq \omega(J)$. The previous inclusion thus implies $\ind_J \ind_{A(P,j)}\leq \ind_{\{x:\ N_j(x)\in  \omega (J)  \}}$.

Finally, we have $P(J)\geq Q(J)$ and $Q(J)\in\T$.  Thus
\begin{align*}
\int_J |g(x)|^{r'} \sum_{j:\ N_j(x)\in \omega(J)} |a_j(x)|^{r'}dx &\lesssim |J|\frac{1}{|J^{(1)}|}\int_{J^{(1)}}|g(x)|^{r'} \sum_{j:\ N_j(x)\in \omega(J)} |a_j(x)|^{r'}dx
\\
&=	|J| \frac{1}{|I_{P(J)}|}\int_{I_{P(J)}} |g(x)|^{r'} \sum_{j:\ N_j(x)\in \omega_{P(J)} } |a_j(x)|^{r'}dx
\\
&\leq |J| \density(\T)^{r'},
\end{align*}
by the definition of density.
\end{proof}
For $\tau\in\{u={\rm up},d={\rm down}\}$ define the functions
\begin{align*}
F_{J,\tau} (x)&\coloneqq \ind_J(x)g(x) \sum_{\substack{P \in\T_\tau \\ I_P\supsetneq J} }\epsilon_P \ip f,w_{P_d},w_{P_d}(x) a_P(x) 
\\
&= \ind_J(x)g(x) \sum_{\substack{P \in\T_\tau \\ I_P\supsetneq J} } \sum_{j=1} ^{K(x)} \epsilon_P\ip f,w_{P_d},w_{P_d}(x) a_j(x) \ind_{A(P,j)} (x)	.
\end{align*}
Since every tree $\T$ can be written as a union of its up-part $\T_u$ and its down-part $\T_d$, we have the estimate
\begin{align*}
\norm gC_\T f.L^{r'}(\RR). \leq\big( \sum_{J\in\mathcal J}\norm F_{J,d}.L^{r'}(\RR). ^{r'}\big)^\frac{1}{r'}+\big( \sum_{J\in\mathcal J}\norm F_{J,u}.L^{r'}(\RR). ^{r'}\big)^\frac{1}{r'}.
\end{align*}
We estimate the two terms appearing in the previous sum separately.
\begin{lemma} We have
\begin{align*}
\big( \sum_{J\in\mathcal J}\norm F_{J,d}.L^{r'}(\RR). ^{r'}\big)^\frac{1}{r'}\lesssim |I_T|^\frac{1}{r'} \density(\T) \size(\T).
\end{align*}
\end{lemma}

\begin{proof} The function $F_{J,d}$ can be written as
\begin{align*}
F_{J,d}(x)=\ind_J(x)g(x) \sum_{\substack{P \in\T_d \\ I_P\supsetneq J} } \sum_{j=1} ^{\infty} \epsilon_P\ip f,w_{P_d},w_{P_d}(x) a_j(x) \ind_{[j,\infty)}(K(x)) \ind_{A(P,j)} (x).
\end{align*}
Now consider two pairs $(j,P)\neq (j',P')$ appearing in the previous sum, with $ j,j' \geq 1 $ and $P,P'\in \T_d$ such that $I_P,I_{P'}\supsetneq J$. We claim that $A(P,j)\cap A(P',j')=\emptyset$. Indeed, since $\T_d$ is a down-tree we have $\omega_{T_d}\subseteq \omega_{P_d}\cap \omega_{P' _d}\neq \emptyset$ and thus the bigger interval contains the smaller. For example we have $\omega_{P' _d} \subseteq \omega_{P_d}\Rightarrow \omega_{P'} \subseteq \omega_{P}.$  

If $j' <j\Leftrightarrow j' \leq j-1$ and $x\in A(P,j)$ then we have $(x,N_j(x))\in P_u $ and $(x,N_{j-1}(x)) \notin \mathring P$. Since $N_j(x)>N_{j-1}(x)$ we get that necessarily $N_{j-1}(x)\leq \min  \omega_P$. Thus
\begin{align*}
N_{j'}(x)\leq N_{j-1}(x) \leq \min\omega_{P}\leq \min \omega_{P'}\Rightarrow N_{j'}(x)\notin \omega_{P' _u}\Rightarrow x\notin A(P',j').
\end{align*}

Suppose now that $j'> j\Leftrightarrow j\leq j'-1$. If $x\in A(P',j')$ then $(x,N_{j'}(x))\in P' _u$ and $(x,N_{j'-1}(x))\notin \mathring P' $. Since $N_{j'}(x)\geq N_{j'-1}(x)$ we must have $N_{j'-1}(x)\leq \min \omega_{P'}$. Thus
\begin{align*}
N_j(x)\leq N_{j'-1}(x)\leq \min \omega_{P'} =\min \omega_{P' _d}< \min \omega_{P_u}
\end{align*}
since $\omega_{P_d '} \subseteq \omega_{P_d}$. Thus $N_j(x)\notin \omega_{P_u}$ which implies that $x\notin A(P,j).$
In every case we get that $A(P,j)\cap A(P',j')\neq \emptyset \Rightarrow j=j'$. However all the up-parts $P_u$, $P\in \T_d$, are disjoint since $\T_d$ is a down-tree, so we cannot have $N_j(x) \in \omega_{P_u}\cap\omega_{P' _u}$ with $P\neq P'$. We conclude that $A(P,j)\cap A(P',j')\neq \emptyset \Rightarrow (P,j)=(P',j')$ as claimed.

The disjointness property of the $A(P,j)$'s implies that
\begin{align*}
|F_{J,d}(x)| &\leq |g(x)| \sup_{\substack{P \in\T_d \\ I_P\supsetneq J} }\sup_{1\leq j\leq K(x)}\frac{|\ip f ,w_{P_d},|}{|I_P|^\frac{1}{2} }|a_j(x)| \ind_J(x) \ind_{A(P,j)}(x)
\end{align*}
Furthermore, by the definition of size it is not hard to see that $|\ip f,w_{P_d},/|I_P|^\frac{1}{2}\leq \size(\T)$, by testing the definition against a tree consisting of a single bitile $P$. Combining this estimate with (ii) of Lemma~\ref{l.aux} and the previous display we get that for any fixed $x$
\begin{align*}
|F_{J,d}(x)|&\leq \size(\T) |g(x)| \sup_{1\leq j \leq K(x)} |a_j(x)|\ind_{\{y:\ N_j(y)\in  \omega(J)   \}}(x)
\\
&\leq \size(\T) |g(x)|\  \bigg(\sum_{j:\ N_j(x)\in \omega(J)} |a_j(x)|^{r'} \bigg)^\frac{1}{r'}.
\end{align*}
Integrating the previous estimate raised to the power $r'$ yields
\begin{align*}
\norm F_{d,J}.L^{r'}(\RR). ^{r'} &\lesssim \size(\T) ^{r'}\int_J |g(x)|^{r'} \sum_{j:\ N_j(x)\in \omega(J)} |a_j(x)|^{r'} dx \lesssim |J| \size(\T) ^{r'} \density(\T)^{r'},
\end{align*}
by Lemma \ref{l.aux}, (iii).
Since the collection $\mathcal J$ partitions $I_T$, summing in $J\in\mathcal J$ gives
\begin{align*}
\big(\sum _{J\in\mathcal J} \norm F_{d,J}.L^{r'}(\RR). ^{r'}\big)^\frac{1}{r'} \lesssim \size(\T)  \density(\T) \big(\sum_{J\in\mathcal J} |J| \big)^\frac{1}{r'} =|I_T|^\frac{1}{r'} \size(\T) \density(\T), 
\end{align*}
as we wanted.
\end{proof}
The proof for the up-part is more involved:
\begin{lemma} We have
\begin{align*}
\big( \sum_{J\in\mathcal J}\norm F_{J,u}.L^{r'}(\RR). ^{r'}\big)^\frac{1}{r'}\lesssim |I_T|^\frac{1}{r'} \density(\T) \size(\T).
\end{align*}
\end{lemma}

\begin{proof}  We fix some $J\in\mathcal J$ and $x\in J$ so that $F_{J,u}(x)\neq 0$. We use Lemma \ref{l.haar} in order to write the function $F_{J,u}$ in the form:
\begin{align*}F_{J,u}(x)& =\ind_J(x)g(x)w_{T_u } ^\infty (x)  \sum_{\substack{1\leq j\leq K(x) \\ j:\ N_j(x)\in\omega(J)}}  \sum_{\substack{P \in\T_u \\ I_P\supsetneq J} } \epsilon_P\epsilon_{PT}\ip f,w_{P_d}, h_{I_P}(x) a_j(x) \ind_{A(P,j)} (x),
\end{align*}
where $\epsilon_P\epsilon_{PT}\in\{-1,+1\}$ and $w_{T_u} ^\infty$ is unimodular and depends only on the upper tile of the top $T$ of $\T$. Now for every $1\leq j\leq K(x)$ with $N_j(x)\in \omega(J)$ , consider the inner sum
\begin{align*}
 \sum_{\substack{P \in\T_u \\ I_P\supsetneq J} }\epsilon_P\epsilon_{PT}\ip f,w_{P_d}, h_{I_P}(x) a_j(x) \ind_{A(P,j)} (x).
\end{align*}
Let us consider $j,x$ such that $\ind_{A(P,j)}(x)\neq 0$ in the previous sum and examine which bitiles $P\in\T_u$ contribute to it. For such bitiles we must have $N_{j-1}(x)\notin \omega_P$ and $N_j(x)\in\omega_{P_u}$. Now the frequency intervals $\omega_{P_u}$, $P\in\T_u$, all contain the top interval $\omega_{T_u}$ and thus they are nested.  This nestedness property implies that if $N_j(x)\in \omega_{P' _u}$ is satisfied for some $P'\in\T_u$ then it will also be satisfied for all $P\in\T_u$ with $\omega_{P_u}\supset \omega_{P' _u}$. Likewise, all the $\omega_P$'s of bitiles $P\in\T_u$ that contribute to the sum are nested since they all contain the top interval $\omega_{T}$. Thus, if $N_{j-1}(x)\notin \omega_{P'}$ for some $P'\in\T_u$ then the same condition will also be satisfied for all $P\in\T_u$ with $\omega_P\subset \omega_{P'}$. We conclude that, for each $x,j$, the bitiles that contribute to the sum are nested, their frequency intervals all contain some minimum frequency interval $\omega_{x,j}$ and are contained in some maximum frequency interval $\Omega _{x,j}$. Now observe that the time intervals of these bitiles are also nested since they all contain $J$. Since the area of each bitile is fixed we conclude that for every $J\in\mathcal J$, $x\in J$ and $1\leq j \leq K(x)$, there are some dyadic intervals $ I_{x,j} ^{\rm small}, I_{x,j} ^{\rm large}$ such that $J\subset  I_{x,j} ^{\rm small} \subsetneq I_{x,j} ^{\rm large}$
and
\begin{align*}
\Abs{\sum_{\substack{P \in\T_u \\ I_P\supsetneq J} }\epsilon_P \epsilon_{PT}\ip f,w_{P_d}, h_{I_P}(x) a_j(x) \ind_{A(P,j)} (x)}=\Abs{\sum_{\substack {P\in\T_u \\ I_{x,j} ^{\rm small}  \subsetneq  I_P \subseteq I_{x,j} ^{\rm large} } }\epsilon_P  \epsilon_{PT}\ip f,w_{P_d}, h_{I_P}(x) a_j(x) }.
\end{align*}
In fact we will have that $I_{x,j} ^{\rm small}$ is the time interval corresponding to $\Omega_{x,j}$ and that $I_{x,j} ^{\rm large}$ is the time interval corresponding to $\omega_{x,j}$. From this it is also not hard to see that we also have the property $ I_{x,j} ^{\rm large}\subsetneq I_{x,j+1} ^{\rm small}$ for each $j$.
Based on these observations and notations we can now estimate
\begin{align*}
&\Abs{\sum_{\substack{P \in\T_u \\ I_P\supsetneq J} }\epsilon_P \epsilon_{PT}\ip f,w_{P_d}, h_{I_P}(x) a_j(x) \ind_{A(P,j)} (x)}=\Abs{\sum_{\substack {P\in\T_u \\ I_{x,j} ^{\rm small}  \subsetneq  I_P \subseteq I_{x,j} ^{\rm large} } }\epsilon_P  \epsilon_{PT}\ip f,w_{P_d}, h_{I_P}(x) a_j(x) }.
\\
&\quad\quad \leq \Abs{ a_j(x) w_{T_u} ^\infty(x)(\Exp_{I_{x,j} ^{\rm small }}-\Exp_{I_{x,j} ^{\rm large } }	)\big(\sum_{P\in\T_u}\epsilon_P\epsilon_{PT}\ip f,w_{P_d}, h_{I_P} \big)(x)}
\\
&\quad\quad = \Abs{a_j(x) (\Exp_{I_{x,j} ^{\rm small }}-\Exp_{I_{x,j} ^{\rm large } }	)\big(w_{T_u} ^\infty \sum_{P\in\T_u}\epsilon_P \ip f,w_{P_d},w_{P_d} \big)(x)},
\end{align*}
since $|w_{T_u} ^\infty |= 1$. The expectation operator $\Exp_I$ is defined  in \eqref{e.condk}.

Remember that $\Exp_{\ell}f$ denotes the martingale of dyadic averages of $f$ with respect to dyadic intervals of length $2^\ell$. Summing in $j\in[1,K(x)]$ for which $N_j(x)\in\omega(J)$ and using H\"older's inequality we get
\begin{align*}
|F_{J,u}(x)| & \leq \ind_J (x) |g(x)|  \sum_{\substack{1\leq j\leq K(x) \\ j:\ N_j(x)\in\omega(J)}}  \Abs{a_j(x) (\Exp_{I_{x,j} ^{\rm small }}-\Exp_{I_{x,j} ^{\rm large } }	)\big(w_{T_u} ^\infty \sum_{P\in\T_u} \epsilon_P \ip f,w_{P_d},w_{P_d} \big)(x)}
\\
&\leq \ind_J(x)|g(x)| \ \bigg(	\sum_{\substack{1\leq j\leq K(x) \\ j:\ N_j(x)\in\omega(J)}}  |a_j(x)|^{r'}\bigg)^\frac{1}{r'} \bigg( \sum_{\substack{1\leq j\leq K(x) \\ j:\ N_j(x)\in\omega(J)}} \Abs{  (\Exp_{\ell_j}-\Exp_{\ell' _j})  (w_{T_u} ^\infty \tilde f) (x)}^r \bigg)^\frac{1}{r},
\end{align*}
where $\ell_1,\ell_2,\ldots,\ell_1 ',\ell_2 ',\ldots,$ are integers with $ |J|\leq 2^{\ell_1}<2^{\ell_1 {'}}<2^{\ell_2}<2^{\ell_2 {'}}<\ldots<2^{\ell_j }<2^{\ell_j {'}}<\ldots , $ and $\tilde f\coloneqq  \sum_{P\in\T_u}\epsilon_P \ip f,w_{P_d},w_{P_d} $. Taking the supremum over all such choices of integers $\ell_j,\ell_j  {'} $ for $1\leq j \leq K$ and all positive integers $K$ and integrating over $J$ gives the estimate
\begin{align*}
\int_J |F_{J,u}(x)|^{r'} dx & \leq \int_J |g(x)|^{r'} \sum_{j: \ N_j(x)\in \omega(J)}|a_j(x)|^{r'} \| \, \Exp_k( w_{T_u} ^\infty \tilde f) (x) \|_{\mathcal V^r_{2^k\geq |J|}} ^{r'} dx.
\end{align*}
Here we denote 
\begin{align*}
\|b_k\|_{\mathcal V^r _{2^k\geq A}}\coloneqq \sup_K \sup_{\substack{\ell_0<\ell_1<\cdots<\ell_K \\ 2^{\ell_j}\geq A,\ j=0,\ldots,K}} \big(	\sum_{j=1} ^K \|\beta_{\ell_{j+1}}-\beta_{\ell_j}\|^r\big)^\frac{1}{r}.
\end{align*} 

The function  $\| \Exp_k w_{T_u} ^\infty \tilde f (\cdot) \|_{\mathcal V^r_{2^k\geq |J|}}$ is constant on $J$ thus 
\begin{align*}
\|F_{J,u}\|_{L^{r'}(\RR)} ^{r'}& \leq \int_{J} |g(x)|^{r'}  \sum_{j: \ N_j(x)\in \omega(J)}|a_j(x)|^{r'} \bigg(\frac{1}{|J|}\int_J   \| \, \Exp_k (w_{T_u} ^\infty \tilde f) (z) \|_{\mathcal V^r_{2^k \geq |J|}} dz\bigg) ^{r'} dx
\\
&\leq |J| \density(\T)^{r'}  \inf_{ J} \big[ M\big (  \|\, \Exp_k (w_{T_u} ^\infty \tilde f)(\cdot)  \|_{\mathcal V^r } \big) \big]^{r'}
\\
&\leq \density(\T)^{r'} \int_J  \big[M\big(  \| \, \Exp_k (w_{T_u} ^\infty \tilde f)(\cdot)  \|_{\mathcal V^r }\big)(y) \big]^{r'} dy,
\end{align*}
where in the second inequality we also used Lemma \ref{l.aux}, (iii). Here we remember that $M$ is the dyadic maximal operator defined in \S~\ref{s.notation}. Now we sum over $J\in\mathcal J$ and use the boundedness of $M$ on $L^{r'}$ to get
\begin{align*}
\big(\sum_{J\in\mathcal J}\|F_{J,u}\|_{L^{r'}(\RR)} ^{r'}\big)^\frac{1}{r'}&\leq \density(\T)\bigg( \int_{I_T}  \big[M(  \|\, \Exp_k (w_{T_u} ^\infty \tilde f)(\cdot)  \|_{\mathcal V^r }(y) \big]^{r'} dy \bigg)^\frac{1}{r'}
\\
&\lesssim \density(\T)\bigg( \int   \big[  \| \, \Exp_k (w_{T_u} ^\infty \tilde f)(y)  \|_{\mathcal V^r }  \big]^{r'} dy \bigg)^\frac{1}{r'}
\\
& \leq  \density (\T) \|\tilde f \|_{L^{r'}(\RR)}.
\end{align*}
where in the last inequality we have used the vector-valued L\'epingle inequality from Proposition \ref{p.leptile}. Observe that the use of Proposition~\ref{p.leptile} is allowed since $r' < q'\leq 2 \leq q<r$. Now Lemma~\ref{l.haar} allows us to write 
\begin{align*}
	\|\tilde f \|_{L^{r'}(\RR)}&=\Norm \sum_{P\in\T_u}\epsilon_P \ip f,w_{P_d},w_{P_d}. L^{r'}(\RR).=\Norm \sum_{P\in\T_u}\epsilon_P \ip f w_{T_u} ^\infty,h_{I_P},h_{I_P}. L^{r'}(\RR).
	\\
	&\lesssim \Norm \sum_{P\in\T_u}\ip f w_{T_u} ^\infty,h_{I_P},h_{I_P}. L^{r'}(\RR). ,
\end{align*}
the last inequality following by the UMD property of $X$. By another use of Lemma~\ref{l.haar} and H\"older's inequality we have
\begin{align*}
 \Norm \sum_{P\in\T_u}\ip f w_{T_u} ^\infty,h_{I_P},h_{I_P}. L^{r'}(\RR).&=|I_T|^\frac{1}{r'} \bigg(\frac{1}{|I_T|} \int_{I_T}\Abs{\sum_{P\in\T_u}  \ip f,w_{P_d},w_{P_d}(x) }^{r'}dx \bigg)^\frac{1}{r'}
\\
&\leq |I_T|^\frac{1}{r'} \bigg(\frac{1}{|I_T|} \int_{I_T}\Abs{\sum_{P\in\T_u}  \ip f,w_{P_d},w_{P_d}(x) }^{q}dx \bigg)^\frac{1}{q}\leq |I_T|^\frac{1}{r'}\size(\T)
\end{align*}
by the definition of size. Combining the last three displays we get
\begin{align*}
	\big(\sum_{J\in\mathcal J}\|F_{J,u}\|_{L^{r'}(\RR)} ^{r'}\big)^\frac{1}{r'}&\lesssim \density(\T)\size(\T)|I_T|^\frac{1}{r'}
\end{align*}
which is the desired estimate.
\end{proof}

\section{The size and density lemmas} \label{s.sensity}Let $X$ be a UMD Banach space with tile-type $q\geq 2$. In this section we recall the standard selection algorithms in terms of density and size. In terms of density we have:
\begin{lemma}[Density lemma]\label{l.density} Let $\P$ be a finite collection of bitiles and $\delta>0$. Define $\density$ with respect to some function $g:\RR\to X^*$ with $|g|\leq \ind_E$, where $E\subset \RR$ is a measurable set of finite measure, and a sequence $\{a_j(x)\}_j$ with $\sum_j |a_j(x)|^{r'}=1$. For a given $\Delta>0$, there exists a decomposition
\begin{align*}
\P=\P_{\operatorname{sparse}}\cup \bigcup_j \T_j,
\end{align*} 
where each $\T_j$ is a tree,
\begin{equation*}
   \density(\P_{\operatorname{sparse}})\leq  \Delta,
\end{equation*}
and for all dyadic $J$:
\begin{equation*}
  \sum_{j:I_{T_j}\subset J} |I_{T_j}| \leq \Delta ^{-r'}|E\cap J|.
\end{equation*}
\end{lemma}
\begin{remark}\label{rem:L1andBMO}
The last estimate encodes different types of information. Letting $J$ increase to $\RR$, it shows that
\begin{equation*}
  \Big\|\sum_j \ind_{I_{T_j}}\Big\|_1=
  \sum_{j} |I_{T_j}| \leq \Delta^{-r'}|E|. 
\end{equation*}
On the other hand, it also shows that
\begin{equation*}
  \Big\|\sum_j \ind_{I_{T_j}}\Big\|_{BMO}
  \lesssim\sup_J\frac{1}{|J|}\sum_{j:I_{T_j}\subset J} |I_{T_j}|
  \leq \Delta^{-r'}
  \sup_J\frac{|E\cap J|}{|J|},
\end{equation*}
where the supremum is over all dyadic $J$ that contain at least one $I_{T_j}$.
\end{remark}
\begin{proof} We choose $\P_{\operatorname{sparse}}$ to be as big as possible
\begin{align*}
\P_{\operatorname{sparse}}\coloneqq\bigg\{P\in\P: \sup_{P'\geq P}\frac{1}{|I_{P'}|}\int_{I_{P'} } |g(x)|^{r'} \sum_{j:\ N_j(x)\in\omega_{P'}}|a_j(x)|^{r'}dx \leq 
  \Delta^{r'} \bigg\}.
\end{align*}
For every $P\in \P\setminus \P_{\operatorname{sparse}}$ there exists a bitile $P'$ such that
\begin{align*}
|I_{P'}|\leq \Delta^{-r'} \int_{I_{P'} } |g(x)|^{r'} \sum_{j:\ N_j(x)\in\omega_{P'}}|a_j(x)|^{r'}dx  .
\end{align*}
Among the chosen bitiles $P'$ let us call $\{T_k\}_k$ the bitiles that are maximal with respect to the partial order `$\leq$'. Now set
\begin{align*}
\T_k \coloneqq \{P\in\P\setminus\P_{\operatorname{sparse}}:P\leq T_k\}.
\end{align*}
It is clear that $\P\setminus \P_{\operatorname{sparse}}=\cup_k \T_k$. We have
\begin{align*}
\sum_{k: I_{T_k}\subset J} |I_{T_k}| 
   & \leq  \Delta^{-r'}
   \sum_{k: I_{T_k}\subset J} \sum_j \int_{I_{T_k}}|g(x)|^{r'} |a_j(x)|^{r'} \ind_{\{y:\ N_j(y)\in \omega_{T_k}\}}(x)dx  \\
  &\leq   \Delta^{-r'} 
  \sum_j \sum_{k: I_{T_k}\subset J} \int_{I_{T_k}\cap E\cap \{y:\ N_j(y)\in \omega_{T_k}\} }  |a_j(x)|^{r'} dx.
\end{align*}
Now, for $j$ fixed, the sets $I_{T_k}\cap E\cap \{y:\ N_j(y)\in \omega_{T_k}\}$, $k\in\Z$, are all contained in $E \cap J $ and are pairwise disjoint for different $k$'s. Indeed if two of them intersected, say for $k_1\neq k_2$, then the corresponding bitiles $T_{k_1},T_{k_2}$ would also intersect, which contradicts their maximality. Summing first in $k$ and then in $j$ we get
\begin{align} \label{e:BMOreturn}
\sum_{k:I_{T_k}\subset J} |I_{T_k}|& \leq \Delta^{-r'} 
  \sum_j \int_{E \cap J}|a_j(x)|^{r'}dx
\leq \Delta^{-r'} |E  \cap J |,
\end{align}
since $\sum_j |a_j(x)|^{r'}=1$.  
\end{proof}

For the selection by size we prove a version of the standard size selection algorithm:
\begin{lemma}[Size lemma]\label{l.size} Let $\P$ be a finite collection of bitiles and $X$ a Banach space of tile-type $q\geq 2$. Define $\size$ with respect to some function $f\in L^q(\RR; X)$. For a given $\varsigma>0$, there exists a disjoint decomposition
\begin{align*}
\P=\P_{\operatorname{small}}\cup \bigcup_j \T_j,
\end{align*} 
where each $\T_j$ is a tree,
\begin{equation*}
  \size(\P_{\operatorname{small}})\leq   \varsigma,
\end{equation*}
and for dyadic intervals $J$:
\begin{equation*}
 \sum_{j: I_{T_j}\subset J} |I_{T_j}| \lesssim \varsigma^{-q} \|f{ \ind_J}\|_{L^q(\RR;X)} ^q
 {\leq\varsigma^{-q}|F\cap J|\quad\text{if}\quad |f|\leq \ind_F.}
\end{equation*}
\end{lemma}
\begin{remark}
As in Remark~\ref{rem:L1andBMO}, the last estimate implies
\begin{equation*}
  \Big\|\sum_j \ind_{I_{T_j}}\Big\|_1\lesssim \varsigma^{-q}|F|,\qquad
  \Big\|\sum_j \ind_{I_{T_j}}\Big\|_{BMO}\lesssim \varsigma^{-q}\sup_J\frac{|F\cap J|}{|J|},
\end{equation*}
where the supremum is over all dyadic $J$ that contain at least one $I_{T_j}$.
\end{remark}
The selection algorithm is contained in \cite{HL}*{Proposition 6.1} but we briefly outline it here for the reader's convenience.
\begin{proof} For every tree $\T$ we set
	\begin{align*}
		\Delta(\T)^q \coloneqq \frac{1}{|I_T|} \int \Abs {\sum_{P\in\T_u}\ip f,w_{P_d},w_{P_d}(x) }^q dx,
	\end{align*}
where $T$ is a top of $\T$. We iterate the following selection algorithm. Consider all maximal trees inside $\P$  with $\Delta( \T)>\varsigma$. 
Among them let $ \T_1$ be one with top $ T_1$ whose frequency interval $\omega_{ T_1}$ has minimal center. Replace $\P$ by $\P\setminus\T _1$ and iterate. When no trees can be selected any longer the remaining collection, $\P_ {\operatorname{small} }$, by definition satisfies $\size(\P _ { \operatorname{small} })\leq \varsigma$. 
Let $\{ \T_j \}_j$ be the selected trees. The top time intervals ${I_{T_j}}$ of the selected trees can be thus estimated by
\begin{align*}
	\sum_{j: I_{T_j}\subset J} |I_{T_j}| \leq \frac{1}{\varsigma^q} \sum_{j: I_{T_j}\subset J} \NOrm \sum_{P\in \T_{j,u} }\ip f,w_{P_d},w_{P_d} .L^q(\RR;X)  .^q,
\end{align*}
and we can replace $f$ by $f\ind_J$, since $w_{P_d}$ is supported on $I_P\subset I_{T_j}\subset J$. The collection $\mathcal T\coloneqq \{\T_{j,u}\}_j$ of the selected up-trees is a \emph{good} collection by construction. Thus the tile-type property of $X$ implies that the sum on the right hand side of the previous estimate can be estimated by $\|f\ind_J\|_{L^q(\RR;X)} ^q$. We get
\begin{align*}
	\sum_{j:I_{T_j}\subset J} |I_{T_j}| \lesssim \varsigma^{-q} \|f\ind_J\|_{L^q(\RR;X)} ^q  
\end{align*}
as desired.
\end{proof}
Iterating the density and size lemmas we can write any finite collection as a union of trees.

\begin{lemma}\label{l.sensity}
Let $\P$ be a finite collection of bitiles, and define  $\operatorname{size}$ with respect to $f:\R_+\to X$ with $\abs{f}\leq \ind_F$ and  $\operatorname{density}$ with respect to $g:\R_+\to X^*$ with $\abs{g}\leq\ind_E$. Then $\P$ admits a decomposition
\begin{align*}
\P = \bigcup_{n\in\Z} \bigcup_j \T_{n,j} \cup \P_{\operatorname{residual}},
\end{align*}
such that
\begin{align*}
\density(\cup_j\T_{n,j}) \leq 2^\frac{n}{r'}|E|^{1/r'},\quad \size(\cup_j \T_{n,j}) \leq 2^\frac{n}{q}|F|^\frac{1}{q}, \quad \sum_{j}|I_{\T_{n,j}}|\lesssim 2^{-n}
\end{align*}
and $ \density( \P_{\operatorname{residual}})=\size(\P_{\operatorname{residual}})=0$.

The following bounds are available under additional assumptions:
\begin{enumerate}
  \item\label{it:smallF} If $\inf_{x\in I_P}M(\ind_F)(x)\leq|F|/|E|\leq 1$ for all $P\in\P$, then
\begin{equation*}
  \Big\|\sum_j \ind_{I_{\T_{n,j}}}\Big\|_{BMO}\lesssim 2^{-n}|E|^{-1}.
\end{equation*}
  \item\label{it:bigF} If $\inf_{x\in I_P}M(\ind_E)(x)\leq|E|/|F|\leq 1$ for all $P\in\P$, then
\begin{equation*}
  \Big\|\sum_j \ind_{I_{\T_{n,j}}}\Big\|_{BMO}\lesssim 2^{-n}|F|^{-1}.
\end{equation*}
\end{enumerate}
\end{lemma}

\begin{proof} Since $\P$ is a finite collection of bitiles there exists some positive integer $n_o$ such that
\begin{align*}
\density(\P')\leq 2^{n_o/r'}|E|^{ 1/r'}, \quad  \size(\P') \leq 2^{n_o/q} |F|^\frac{1}{q} 
\end{align*}
We apply density lemma with  $\Delta=2^{(n_0-1)/r'}|E|^{1/r'}$ to write
\begin{align*}
 \P' =\P' _1\cup \bigcup_j \T_j  
\end{align*}
with 
\begin{align*}
\density(\P' _1) \leq 2^{(n_o-1)/r'}|E|^{1/r'}\quad\text{and}\quad \sum_{j} |I_{T_j}|\leq    (2^\frac{n_o-1}{r'}|E|^{ 1/r'})^{-r'}|E|= 2\cdot 2^ {-n_o}.
\end{align*}
We also have that
\begin{equation*}
  \Big\|\sum_j \ind_{I_{T_j}}\Big\|_{BMO}
  \lesssim (2^\frac{n_o-1}{r'}|E|^{ 1/r'})^{-r'}\sup_J\frac{|E\cap J|}{|J|}\lesssim 2^{-n_o}|E|^{-1}\sup_J\inf_{x\in J}M(\ind_E)(x),
\end{equation*}
where the supremum is over $J$ that contain at least one $I_{T_j}$, therefore at least one $I_P$ with $P\in\P$. The maximal term is always bounded by $1$, and under assumption \eqref{it:bigF} also by $\abs{E}/\abs{F}$.

We reduce the size of the collection in a similar fashion. We apply the size lemma with $\varsigma=2^{(n_o-1)/q}|F|^\frac{1}{q}$  to $\P' _1$ to write
\begin{align*}
 \P' _1= \P'_{1,1}\cup \bigcup_j\tilde \T_j  
\end{align*}
with 
\begin{align*}
 \size(\P' _{1,1}) &\leq 2^{\frac{n_o-1}{q}}|F|^\frac{1}{q}\quad\text{and}\quad 
\\
  \sum_j |I_{\tilde T_j}| &\lesssim (2^\frac{n_o-1}{q}|F|^\frac{1}{q})^{-q}\|f\|_{L^q(\RR;X)}^q\lesssim 2^{-n_o}.
\end{align*}
Under the additional assumptions, we also have that
\begin{equation*}
  \Big\|\sum_j \ind_{I_{\tilde{T}_j}}\Big\|_{BMO}
  \lesssim (2^\frac{n_o-1}{q}|F|^\frac{1}{q})^{-q}\sup_J\frac{|F\cap J|}{|J|}
  \lesssim 2^{-n_0}|F|^{-1}\sup_J\inf_{x\in J}M(\ind_F)(x),
\end{equation*}
where the supremum is over all dyadic $J$ that contain at least one top $I_{\tilde{T}_j}$, hence at least one $I_P$, $P\in\P$. The maximal term is always bounded by $1$, and under the assumption \eqref{it:smallF} also by $\inf_{x\in J}M(\ind_F)(x)\leq\inf_{x\in I_P}M(\ind_F)(x)\lesssim |F|/|E|$.

Altogether, we find that $\{\T_{n_o,j}\}_j:=\{\T_j\}_j\cup\{\tilde{\T}_j\}_j$ satisfies
\begin{equation*}
  \sum_j|I_{\T_{n_o,j}}|\lesssim 2^{-n_o}
\end{equation*}
and
\begin{equation*}
  \Big\|\sum_j \ind_{I_{\T_{n_o,j}}}\Big\|_{BMO}
  \lesssim   \begin{cases}
    2^{-n_0}(|E|^{-1}+|F|^{-1}\cdot|F|/|E|)  \lesssim 2^{-n_o}|E|^{-1}, & \textup{under assumption \eqref{it:smallF}}, \\
    2^{-n_0}(|E|^{-1}\cdot|E|/|F|+|F|^{-1}) \lesssim 2^{-n_o}|F|^{-1}, & \textup{under assumption \eqref{it:bigF}}. \end{cases}
\end{equation*}
Since the density and size of any subcollection of $\P'$ cannot increase thus we also have
\begin{align*}
 \density(\cup_j \T_{n_o,j})\leq 2^\frac{n_o}{r'}|E|^{ 1/r'} \quad\text{and}\quad\size(\cup_j \T_{n_o,j})\leq 2^\frac{n_o}{q} |F|^\frac{1}{q}.
\end{align*}
We iterate this procedure until the residual collection has density and size equal to $0$.
\end{proof}
\begin{remark} In what follows the collection $ \P_{\operatorname{residual}}$ will be ignored. In fact, every $P\in \P_{\operatorname{residual}}$ can be considered as a tree with a single bitile, and this tree will have both size and density equal to $0$. By the tree lemma these trivial trees do not contribute anything to the estimate for $C_\P$.
\end{remark}
We record some additional size estimates that will allow us to obtain some initial control on the size of the collections we will consider.

\begin{lemma}\label{l.trivialsize} Let $\P$ be any collection of bitiles and define the density with respect to some function $g:\RR\to X^*$ with $|g|\leq \ind_E$ and the size with respect to some function $f:\RR\to X$ with $|f|\leq \ind_{F}$. We have that $\density(\P)\leq 1$ and $ \size(\P)\lesssim 1$.

The following bounds are available under additional assumptions:
\begin{enumerate}
  \item\label{it2:smallF} If $\inf_{x\in I_P}M(\ind_F)(x)\lesssim|F|/|E|\leq 1$ for all $P\in\P$, then
\begin{equation*}
  \size(\P) \lesssim\frac{|F|}{|E|}.
\end{equation*}
  \item\label{it2:bigF} If $\inf_{x\in I_P}M(\ind_E)(x)\lesssim|E|/|F|\leq 1$ for all $P\in\P$, then
\begin{equation*}
  \density(\P) \leq \Big(\frac{|E|}{|F|}\Big)^{1/r'}.
\end{equation*}
\end{enumerate}
\end{lemma}
\begin{proof}
The proof of the density estimate by $1$ is completely trivial while the proof of the size estimate by $1$ is in \cite{HL}*{Lemma 7.1} and relies on the UMD property of $X$. 

Under the assumption \eqref{it2:bigF}, the density satisfies
\begin{equation*}
  \density(\P)\leq\sup_{P\in\P}\sup_{P'\geq P}\Big(\frac{1}{|I_{P'}|}|I_{P'}\cap E|\Big)^{1/r'}
  \leq \sup_{P\in\P}\inf_{x\in I_P}M(\ind_E)(x)^{1/r'}\leq\Big(\frac{|E|}{|F|}\Big)^{1/r'}.
\end{equation*}
Under the assumption \eqref{it2:smallF}, a standard argument using Lemmas \ref{l.haar} and \ref{l.bmosize} (below) give the size bound asserted in this case; for the details of this argument see for example the proof of \cite{HL}*{Lemma 7.3}. 
\end{proof}

The following lemma, which we referred to above, is a $\bmo$-type of estimate and is contained in \cite{HL}*{Lemma 7.2}.

\begin{lemma}\label{l.bmosize} Let $\mathcal J\subseteq \{I\in\mathcal D:\inf_{y\in I} Mf(x)\leq \lambda\}$ be a finite collection of dyadic intervals and $K$ be a dyadic interval. Then
	\begin{align*}
		\NOrm \sum_{\substack{I\in\mathcal J\\I\subseteq K}}\ip f,h_I,h_I.L^P(\RR;X). \leq \lambda |K|^\frac{1}{p},
	\end{align*}
for all $1\leq p<\infty.$
\end{lemma}
\section{Proof of Proposition \ref{p.main}} \label{s.summingup}
We need to prove that for every pair of measurable sets $E,F$ with $|E|,|F|<+\infty$ there exists subsets $E'\subset E$, 
 and $ F'\subset F$ such that  either $ E'=E$ or $ F'=F$, and 
\begin{align*}
\abs {\ip C_\P f , g, }\lesssim |F|^\frac{1}{p} |E|^\frac{1}{p'}
\end{align*}
for $\max(q,p'(q-1))<r<\infty$ and for all $|f|\leq\ind_{F'}$, $|g|\leq \ind_{E'}$. By dilation invariance we can assume that $1<|E| \leq 2$.

\subsubsection*{Case 1: $|F|\leq 1$.}
 Let 
\begin{align*}
G \coloneqq \{M(\ind_F)>4|F|\}.
\end{align*}
Then $|G|\leq \frac{1}{4}$ by the weak $(1,1)$ bound of $M$ and thus $E' \coloneqq E\setminus G$ can be taken as a major subset of $E$. We have
\begin{align*}
\ip C_\P f , g, &=\sum_{P\in\P} \epsilon_P \ip f ,w_{P_d}, \ip w_{P_d}  a_{P}, g, =\sum_{\substack{P\in\P\\I_P\nsubseteq G}} +\sum_{\substack{P\in\P\\I_P\subseteq G}} .
\end{align*}
The second sum vanishes since each $w_{P_d}$ is supported on $I_P\subseteq G$ while $g$ is supported on $E'\subseteq G^c$. Hence it suffices to consider any collection $\P' \subseteq \{P\in \P: I_P\nsubseteq G\}$ and prove the corresponding bound for $C_{\P'}$ in the place of $C_\P$. Observe that for all $P\in\P'$ we have that $I_P\nsubseteq G$ so that $\inf_{x\in I_P}M(\ind_F)(x)\leq \inf_{x\in I_P\setminus G}M(\ind_F)(x)\leq 4|F|$ by the construction of $G$. Thus we are now in the situation that $\inf_{x\in I_P}M(\ind_F)(x)\lesssim |F|\eqsim|F|/|E|$ for all $P\in\P'$, where several useful bounds were obtained in the previous section.

We now apply the decomposition given by Lemma~\ref{l.sensity} to the collection $\P'$. Recalling the bounds $\density(\P')\leq 1$ and $\size(\P')\lesssim|F|$ from Lemma \ref{l.trivialsize}, and the estimates of Lemma~\ref{l.sensity}, we can write
\begin{align*}
\density(\T_{n,j}) \lesssim \min(1,2^\frac{n}{r'}),\quad  \size(\T_{n,j}) \leq \min(|F| ,2^\frac{n}{q} |F|^\frac{1}{q}), \quad \sum_j |I_{T_{n,j}}|\lesssim 2^{-n}.
\end{align*}
Applying the Lemma \ref{l.tree} with $s=1$ we get
\begin{align*}
\abs {\ip C_\P f , g, }&=\ABs{ \int \sum_{P\in \P} \epsilon_P \ip f,w_{P_d},w_{P_d}(x)a_P(x)g(x)dx }
\\
&\leq \sum_{n\in\Z}  \sum_j \int \ABs{ \sum_{P\in \T_{n,j}}\epsilon_P \ip f,w_{P_d},w_{P_d}(x)a_P(x)g(x)dx }
\\
&\lesssim \sum_{n\in\Z}  \sum_j |I_{T_{n,j}}| \size(\T_{n,j}) \density(\T_{n,j}) 
\\
&\lesssim \sum_{n\in\Z} \min( |F| ,2^\frac{n}{q} |F|^\frac{1}{q}) \min (1, 2^\frac{n}{r'} ) \sum_j |I_{T_{n,j}}|
\\
&\lesssim \sum_{n\in\Z} 2^{-n} \min(|F| ,2^\frac{n}{q} |F|^\frac{1}{q}) \min (1, 2^\frac{n}{r'} ).
\end{align*}
We estimate the previous sum as follows:
\begin{align*}
\abs {\ip C_\P f , g, }&\lesssim |F|^\frac{1}{q} \sum_{n:\ 2^n\leq |F|^\frac{q}{q'}  }2^{n(\frac{1}{q}-\frac{1}{r})} + |F|\sum_{n:\  |F|^\frac{q}{q'} < 2^n \leq 1  } 2^{n(\frac{1}{r'}-1)}+ |F| \sum_{n:1< 2^n }2^{-n}
\\
&\lesssim |F|^{1-\frac{q-1}{r}}+ |F|\lesssim |F|^\frac{1}{p} .
\end{align*}
In the last approximate inequality we used $1-\frac{q-1}{r}= \frac1p+\frac{1}{p'}-\frac{q-1}{r}>\frac1p $ and  the hypothesis $|F|\leq 1$.

\subsubsection*{Case 2: $|F|\geq 1$.}  
Let $ G \coloneqq \{M \mathbf 1_{E} > 8 \lvert  F\rvert ^{-1}  \}$, and set $ F' \coloneqq F \setminus G$. By the maximal theorem we can conclude that $ F'$ is a major subset of $ F$.  
Hence it suffices to consider any finite collection of bitiles $\P'\subset \{P\in \P: I_P\not\subset G\}$ and prove the corresponding bound for $C_{\P'}$. Thus the collection $\P'$ satisfies the property that $\inf_{x\in I_{P}}M(\ind_E)\lesssim|E|/|F|\eqsim|F|^{-1}$ for all $P\in\P'$, and several estimates from the previous section become available.

We again apply the decomposition given by Lemma~\ref{l.sensity} to the collection $\P'$.
Lemma~\ref{l.trivialsize} now provides the bounds $\density(\P')\lesssim|F|^{-1/r'}$ and $\size(\P')\lesssim 1$; combined with the estimates of Lemma~\ref{l.sensity}, this leads to
\begin{align*}
  \density(\mathcal T_n) \lesssim \min(|F|^{-1/r'},2^\frac{n}{r'}),\quad  \size(\mathcal T_n) \leq \min(1 ,2^\frac{n}{q}),\quad\mathcal T_n:=\cup_j \T_{n,j},
\end{align*}
and
\begin{equation}\label{e.l1bmo}
  \Big\|\sum_j\ind_{I_{\T_{n,j}}}\Big\|_1=\sum_j|I_{\T_{n,j}}|\lesssim 2^{-n},\quad
  \Big\|\sum_j\ind_{I_{\T_{n,j}}}\Big\|_{BMO}\lesssim 2^{-n}|F|^{-1}.
\end{equation}
Interpolation between the last two estimates gives the further bounds
\begin{equation}\label{e.tops}
  \Big\|\sum_j\ind_{I_{\T_{n,j}}}\Big\|_{\tau}\lesssim 2^{-n}|F|^{-1/\tau'},\quad\forall \tau\in[1,\infty).
\end{equation}

For $n $ fixed we estimate
\begin{align*}
 \Abs{\ip C_{\mathcal{T}_{n }}f,g,}&\leq \sum_j \int\ind_{I_{T_{n,j }}}(x) \Abs{ \sum_{P\in\T_{n,j}} \epsilon_P \ip f,w_{P_d},w_{P_d}(x)g(x)a_P(x)}dx
\\
&\leq \int  \bigg(\sum_j \ind_{I_{T_{n,j}}} (x)\bigg)^\frac{1}{r} \bigg(\sum_j \Abs{ \sum_{P\in\T_{n,j} } \epsilon_P \ip f,w_{P_d},w_{P_d}(x)g(x)a_P(x)}^{r'}\bigg)^\frac{1}{r'}dx
\\
&\leq \NORm  \bigg(\sum_j \ind_{I_{T_{n,j}}} \bigg)^\frac{1}{r} . L^\tau.  \NORm \bigg (\sum_j \Abs{ \sum_{P\in\T_{n,j}} \epsilon_P \ip f,w_{P_d},w_{P_d} g a_P}^{r'}\bigg)^\frac{1}{r'}.L^{\tau'}. \eqqcolon A\cdot B,
\end{align*}
for $\tau\geq r$, which we will eventually choose ``large enough''. We note that the function inside the norm in $B$ is supported in $E$ and $\tau' <r'$. Thus, H\"older's inequality and Lemma \ref{l.tree} imply
\begin{align*}
B& \leq |E|^{\frac{1}{\tau'}-\frac{1}{r'}} \bigg(\sum_j  \NORm \sum_{P\in\T_{n,j}}  \epsilon_P \ip f,w_{P_d},w_{P_d}g a_P.L^{r'}.  ^{r'} \bigg)^\frac{1}{r'}
\\
&\lesssim \size(\mathcal T_n)\density (\mathcal T_n) \bigg(\sum_j |I_{T_{n,j}}|\bigg)^\frac{1}{r'}
\\
&\lesssim  \min( 2^{n/q}|F|^{1/q},1)\min(2^{n/r'},|F|^{-1/r'}) 2^{-n/r'}
  =\min( 2^{n/q}|F|^{1/q}, 2^{-n/r'}|F|^{-1/r'})
\end{align*}
Using \eqref{e.tops}  with $\tau$ replaced by $\tau/r>1$  we estimate $A$ as follows
\begin{align*}
A= \NORm  \bigg(\sum_j \ind_{I_{T_{n,j}}}\bigg)^\frac{1}{r} . L^\tau.  = \NORm  \sum_j \ind_{I_{T_{n,j}}} . L^{\tau/r}. ^\frac{1}{r}  \lesssim
  2^{-n/r}|F|^{1/\tau-1/r}.
\end{align*}
Gathering the estimates we get for all $\tau>r$ that
\begin{align*}
 \Abs{\ip C_{\mathcal{T}_n}f,g,}&\lesssim  2^{-n/r}|F|^{1/\tau-1/r}\min( 2^{n/q}|F|^{1/q}, 2^{-n/r'}|F|^{-1/r'}).
\end{align*}
Summing in  $n\in\Z$ thus gives
\begin{equation*}
  \sum_{n\in\Z} \Abs{\ip C_{\mathcal{T}_n}f,g,}
  \lesssim\sum_{2^n\leq|F|^{-1}}2^{n(1/q-1/r)}|F|^{1/\tau-1/r+1/q}+\sum_{2^n>|F|^{-1}}2^{-n}|F|^{1/\tau-1}
  \lesssim|F|^{1/\tau}\leq|F|^{1/p}
\end{equation*}
by taking $\tau\geq \max(p,r)$, and recalling that $|F|\geq 1$.

This concludes the proof of Proposition \ref{p.main} and thereby of Theorems \ref{t.secondary} and \ref{t.main_new}.

\begin{remark} It is of some importance to note a subtle difference between the way we derive estimates~\eqref{e.l1bmo} in the present paper and the way such estimates have been derived in the scalar case in the literature. For example, in \cite{OSTTW} and \cites{DOLA,DOLA1}, only the tops of trees produced by the size lemma are shown to satisfy the BMO estimate in \eqref{e.l1bmo}. To deal with this problem the standard approach is to collect all the trees produced by the size and density lemmas and ``pass them through'' the size lemma once more. This double application of the size lemma guarantees the BMO estimate. One however needs to argue that after the second application of the size lemma, the $L^1$-estimate persists. This is done by complementing the size lemma with a certain efficiency estimate, stating that if a collection is already a union of trees, $\P=\cup_j \T_j$, and the size lemma decomposes $\P$ into a union of some other trees $\{\tilde \T_k\}_k$, then $\sum_k |I_{\tilde T_k}|\lesssim \sum_j |I_{T_j}|$; that is, the size lemma is shown to be the most efficient algorithm in selecting trees and their tops when decomposing $\P$. In the vector-valued case we haven't managed to produce such an efficiency estimate for the size lemma, which is strongly dependent on the tile-type of the Banach space $X$. Instead, we directly produce BMO estimates, both for trees selected by size, as well as for the ones selected by density.
\end{remark}

\section{Interpolation}\label{s.inter}
In this section we discuss the proof of Theorem \ref{t.main2}. This is a simple interpolation argument between the Hilbert space-valued case for the $r$-variation and the UMD valued case for the $\infty$-variation, that is the main result from \cite{HL}.

\begin{proof}[Proof of Theorem \ref{t.main2}]
Let $X$ be a UMD Banach space of the form $X=[Y,H]_\theta$ for some $0<\theta<1$ and set $q\eqqcolon 2/\theta$. From \cite{HL} we have that the Carleson operator maps $L^p([0,1);X)$ into itself for all $1<p<\infty$. This can be rewritten in the form
\begin{align*}
\norm S_N f. L^p([0,1);\ell^\infty(X)). \lesssim \|f\|_{L^p ([0,1);X)},\quad 1<p<\infty.
\end{align*}
On the other hand we have for any Hilbert space $H$ the following variation norm Carleson theorem:
\begin{align*}
\norm S_N f. L^p([0,1);\mathcal V^s(H)). \lesssim \|f\|_{L^p ([0,1);H)},\quad s>2,\quad  s'<p<\infty.
\end{align*}
This follows for example from Theorem \ref{t.main1} although in the special case of a Hilbert space one could just repeat the scalar proof. We first use the reiteration theorem 
as in \cite{BL}*{Theorem 3.5.3}, to write
\begin{align*}
X=[Y,H]_\theta=[ [ Y,H]_{\delta} ,H  ]_\omega,
\end{align*}
with $\theta=(1-\omega)\delta+\omega$. In our considerations one should think of $\delta\to 0$. Now fix $1<p<\infty$ and interpolate between $L^p([0,1);\ell^\infty([Y,H]_\delta))$ and $L^p([0,1);\mathcal V^s(X))$. As in \cite{PX}*{p. 501}, we have
\begin{align*}
[\ell^\infty([Y,H]_\delta) ,\mathcal V^s(H)]_\omega \subset \mathcal V^{s_\omega}([[Y,H]_\delta,H]_\omega)=\mathcal V^{s_\omega}(X),
\end{align*}
where $\frac{1}{s_\omega}= \frac{1-\omega}{\infty}+\frac{\omega}{s}\Leftrightarrow s_\omega=s/\omega,$ and $\theta=(1-\omega)\delta+\omega$.
From this we get that
\begin{align*}
 \norm S_N f.L^p([0,1);\vr(X)). \lesssim \norm f.L^p([0,1);X).
\end{align*}
whenever $r>2/\theta=q$ and $r>p'q/2$.
\end{proof}

\section*{Acknowledgment} We would like to thank the anonymous referee for an expert reading that resulted to several improvements throughout the paper.

 
 \end{document}